\documentclass[11pt, letterpaper, oneside]{myart}
\usepackage{amssymb} 
\usepackage[all]{xy}
\usepackage{color}
\usepackage{bm}

\definecolor{linkcol}{rgb}{0, 0, 0.8}
\usepackage[ linkcolor=linkcol, citecolor=linkcol, colorlinks=true, hypertexnames=false]{hyperref}

\numberwithin{equation}{section}
\renewcommand{\theequation}{\arabic{section}-\arabic{equation}}

\swapnumbers

\theoremstyle{plain}
\newtheorem{theorem}{Theorem}[section]
\newtheorem{lemma}[theorem]{Lemma}
\newtheorem{proposition}[theorem]{Proposition}
\newtheorem{corollary}[theorem]{Corollary}

\theoremstyle{definition}
\newtheorem{definition}[theorem]{Definition}

\newtheorem{remark}[theorem]{Remark}

\newtheorem{nn}[theorem]{}

\newcommand{\la}{\leftarrow}

\newcommand{\lra}{\longrightarrow}
\newcommand{\hra}{\hookrightarrow}
\newcommand{\ratail}{\rightarrowtail}

\newcommand{\retrmanlabel}[4]{\ensuremath \xymatrix@1{#1\ar@<0.4ex>[r]^-{#3}&#2 \ar@<0.4ex>[l]^-{#4} }}
\newcommand{\retrlabel}[2]{\ensuremath \xymatrix@1{#1\ar@<0.4ex>[r]^-{r}&#2 \ar@<0.4ex>[l]^-{s} }}

\newcommand{\CC}{\mathcal{C}}
\newcommand{\DD}{\mathcal{D}}

\renewcommand{\SS}[1][\bullet]{w\mathcal{S}_{#1}}

\newcommand{\TT}[1][\bullet]{w\mathcal{T}_{#1}}

\newcommand{\TTc}{\mathcal{T}_{\bullet}}
\newcommand{\SSc}{\mathcal{S}_{\bullet}}

\newcommand{\Rfd}{\mathcal{R}^{\it fd}}

\newcommand{\simp}[1][B]{{\mathcal S}(#1)}

\newcommand{\Chfd}{{\mathcal C}h^{\it fd}}

\newcommand{\colim}[1][]{\ifmmode\ifinner {\operatorname{colim}_{#1}}\,\else \underset{#1}{\operatorname{colim}}\, \fi\fi}
\newcommand{\hocolim}[1][]{\ifmmode\ifinner {\operatorname{hocolim}_{#1}}\,\else \underset{#1}{\operatorname{hocolim}}\,\fi\fi}
\newcommand{\holim}[1][]{\ifmmode\ifinner {\operatorname{holim}_{#1}}\,\else \underset{#1}{\operatorname{holim}}\,\fi\fi}

\newcommand{\tQ}{\widetilde{Q}}

\newcommand{\wh}[1][{\mathbb R}]{{\rm Wh}^{#1}} 

\newcommand{\ta}{\tilde{a}}

\newcommand{\ttaus}{\ensuremath\tilde \tau^{s}}
\newcommand{\taus}{\ensuremath\tau^{s}}
\newcommand{\ts}{\ensuremath t^{s}}

\newcommand{\hofib}{\mathrm{hofib} }

\newcommand{\id}{{\ensuremath{\rm id}}}

\newcommand{\R}{\mathbb{R}}
\newcommand{\Q}{\mathbb{Q}}
\newcommand{\Z}{\mathbb{Z}}

\newcommand{\bast}{{\bm \ast}}

\newcommand{\oF}{{\overline F}}

\title{\bf Equivalence of higher torsion invariants}

\author[B. Badzioch]{Bernard Badzioch}
\address[]{Department of Mathematics, University at Buffalo, SUNY, Buffalo, NY}
\email{badzioch@buffalo.edu}

\author[W. Dorabia{\l}a]{Wojciech Dorabia{\l}a}
\address[]{Department of Mathematics, Penn State Altoona, Altoona, PA }
\email{wud2@psu.edu}

\author[J. R. Klein]{John R. Klein}
\address[]{Department of Mathematics, Wayne State University, Detroit, MI}
\email{klein@math.wayne.edu}

\author[B. Williams]{Bruce Williams}
\address[]{Department of Mathematics, University of Notre Dame, IN}
\email{williams.4@nd.edu}

\begin{document}

\begin{abstract}
We show that the smooth torsion of bundles of manifolds constructed by Dwyer, Weiss, and Williams
satisfies the axioms for higher torsion developed by Igusa. As a consequence we obtain that 
the smooth Dwyer-Weiss-Williams torsion is proportional to the higher torsion of Igusa and Klein. 
\end{abstract}

\date{\bf 04/30/2009}
\maketitle

\tableofcontents


\section{Introduction}
\label{INTRO SEC}

Higher torsion invariants are invariants of smooth fiber bundles of manifolds which 
generalize the classical Reidemeister torsion of topological spaces. Just as 
the Riedemeister torsion helps to  distinguish spaces which have the same homotopy 
type but differ in their geometric properties, the purpose of the higher torsion is to aid in
classification of smooth bundles up to fiberwise diffeomorphism, especially in the case when  
bundles are fiberwise homotopy equivalent and thus cannot be set apart  using homotopy 
type invariants. 

The idea that the Reidemeister torsion could be generalized to the setting of smooth bundles 
appeared in the work of Hatcher and Wagoner \cite{HW, Wag}. The task of constructing 
higher torsion, however,  proved challenging and required prior development of such areas 
as local index theory, the theory of parametrized Morse functions, and Waldhausen's algebraic 
$K$-theory of spaces. 

These advancements brought in recent years three independent and very distinct 
constructions of higher torsion.  In  \cite{Bismut-Lott} Bismut and Lott presented an analytic construction. 
The torsion invariant developed by  Igusa and Klein \cite{IgusaBook, KleinMorse} follows 
from a more geometric, Morse theoretical approach. The most recent of the three is the construction 
of smooth torsion described by Dwyer, Weiss, and Williams in \cite{DWW}. 
Its striking feature is that while it captures some information about the smooth structure of a bundle, 
it can be described entirely in terms of homotopy theory. 

These invariants have already proved to be effective tools for studying smooth bundles.
 For example, in \cite{IguICM} Igusa demonstrated that the Igusa-Klein torsion can be used 
 to distinguish  exotic disc bundles constructed by Hatcher,  while Goette \cite{Goette1, Goette2} 
 obtained results of a similar kind in the realm of the Bismut-Lott torsion. The central question 
 have become, however, how the three notions of higher torsion are related to one another.  
 The analytic torsion of Bismut-Lott and the higher torsion of Igusa-Klein are known to coincide 
 (up to a normalization constant) for several classes of bundles for which both of these invariants 
 are defined; see e.g. \cite{Bismut-Goette, Goette1, Goette2}. On the other hand, no results have been 
 known comparing them to the smooth torsion of Dwyer-Weiss-Williams.

In order to provide a general framework for settling the comparison problem Igusa developed in 
\cite{IgusaAx} an axiomatic approach to higher torsion. He showed that any cohomological higher 
torsion invariant which is defined on the class of unipotent bundles (\ref{UNIPOTENT DEF}) and which 
satisfies  the additivity (\ref{ADD AX DEF}) and transfer (\ref{TRANSFER AX DEF}) axioms must coincide 
with the Igusa-Klein torsion up to a couple of scalar constants. 
In \cite{BDW} Badzioch, Dorabia{\l}a and Williams showed that the Dwyer-Weiss-Williams torsion 
(which was originally defined in \cite{DWW} for acyclic bundles) 
can be extended to the class of all unipotent bundles and that 
it yields a cohomological invariant of such bundles. The main result of this paper 
(Theorem \ref{MAIN THM}) is that this cohomological invariant satisfies the axioms of Igusa. 
Combined with the observation that the Dwyer-Weiss-Williams torsion is an exotic invariant 
(\ref{EXOTIC DEF}) we obtain

\begin{theorem}[cf. Theorem \ref{MAIN2 THM}]
\label{MAIN INT THM}
For any $k>0$ the smooth cohomological Dwyer-Weiss-Williams torsion of unipotent bundles 
in degree $4k$ is proportional to the Igusa-Klein torsion in the same degree.
\end{theorem}

Apart from its intrinsic interest, the fact that these two very different constructions yield the same information
about smooth bundles has several useful consequences. On the one hand, the Igusa-Klein torsion is 
more suitable for direct computations than the Dwyer-Weiss-Williams invariant. In fact, while Igusa-Klein 
torsion has been calculated for many classes of bundles, in Section \ref{NONTRIVIALITY SEC} of this paper 
we give the first, to our knowledge, examples of bundles for which the Dwyer-Weiss-Williams torsion 
does not vanish. By our comparison result, however, all computations using Igusa-Klein torsion apply
instantly to the smooth torsion of Dwyer-Weiss-Williams. 

On the other hand, the approach of Dwyer-Weiss-Williams seems to provide a better understanding of  the 
information carried by torsion of a bundle. Recall, for example, that in the classical setting 
the Reidemeister torsion is an invariant related to the Whitehead torsion 
of a homotopy equivalence of spaces. 
Results of \cite{DWW} imply that, analogously, one can define a smooth parametrized Whitehead torsion 
as an invariant associated to fiberwise homotopy equivalence $f$ of smooth bundles. 
This invariant vanishes when $f$ is homotopic to a diffeomorphism of bundles. 
Moreover, the opposite is also true: if the smooth parametrized Whitehead torsion of 
$f$ is trivial then $f$ (after an appropriate stabilization) is homotopic to a diffeomorphism.  
The smooth torsion of Dwyer-Weiss-Williams is an absolute invariant associated to 
the parametrized Whitehead torsion. Viewed from this perspective the Dwyer-Weiss-Williams torsion 
of a bundle $p$ can be seen as an obstruction to existence of a diffeomorphism between $p$ and 
a product bundle.
 
Another application of Theorem \ref{MAIN INT THM}  comes as 
a consequence of Theorem \ref{LH TRANSFER THM} of this paper. For unipotent bundles 
$p\colon E\to B$ and $q\colon D\to E$ the transfer axiom of Igusa (\ref{TRANSFER AX DEF}) 
gives a formula expressing torsion of the bundle $pq$ in terms of torsions of $p$ and $q$. 
This formula is known to hold for the Igusa-Klein torsion under the assumption that  $q$ 
is an oriented linear sphere bundle.  Our Theorem \ref{LH TRANSFER THM} shows that for 
the Dwyer-Weiss-Williams torsion the same formula holds in a more general setting, whenever 
the bundle $q$ satisfies the assumptions of the Leray-Hirsch theorem (we say then that $q$ 
is a Leray-Hirsch bundle, see Definition \ref{LH BUNDLE DEF}). Theorem \ref{MAIN INT THM} 
immediately implies that this is true for the Igusa-Klein torsion as well:

\begin{theorem}
\label{IK LH TRANSFER THM}
For $k>0$ let $t^{\it IK}_{4k}$ be the Igusa-Klein torsion invariant in the degree 4k. 
Let $p\colon E\to B$ be a unipotent bundle, and let $q\colon D\to E$ be a Leray-Hirsch bundle 
with fiber $F$. Then we have
$$t^{\it IK}_{4k}(pq)=\chi(F)t^{\it IK}(p)+ tr^{E}_{B}(t^{\it IK}_{4k}(q))$$
where $\chi(F)$ is the Euler characteristic of $F$ and $tr^{E}_{B}\colon H^{\ast}(E; \R)\to H^{\ast}(B; R)$
is the transfer homomorphism associated to $p$.
\end{theorem}
\noindent This property extends computability of the Igusa-Klein invariant.

\begin{nn}{\bf Organization of the paper.}
In Section \ref{AXIOMS SEC} we give a summary of Igusa's axiomatic description of higher 
torsion of smooth bundles. Since our work on the smooth torsion is heavily dependent on 
the language of Waldhausen categories we give a brief overview of the relevant notions and construction 
in Section \ref{WAL SEC}. Section \ref{DWW SEC} describes the main steps in the construction 
of the smooth torsion of Dwyer-Weiss-Williams for unipotent bundles. We also state there precisely 
the main results of this paper (Theo\-rems \ref{MAIN THM} and \ref{MAIN2 THM}). 
The focus of Section \ref{ADD TORSION SEC} is Theorem \ref{ADD THM} which describes 
the additive property of the smooth torsion. 
As we have already mentioned the transfer axiom of Igusa provides a relationship between 
the torsion of bundles $p$, $q$, and $pq$.
In Section \ref{SECONDARY TRANSFER SEC } we show that for the smooth torsion such relationship 
can be expressed using a ``secondary transfer'' map (Theorem \ref{SEC TRANSFER THM})
and then, in Section \ref{TRANSFER AXIOM SEC}, we verify that Igusa's axiom can be derived 
from our formula. Finally, in Section \ref{NONTRIVIALITY SEC} we demonstrate that the smooth torsion
is a non-trivial invariant of bundles. 

A part of the construction of the smooth torsion of unipotent bundles, and consequently a part of 
our arguments, involves the passage from chain complexes to their homology on 
the level of the algebraic $K$-theory. The material related to this topic is gathered in 
the \hyperref[APP SEC]{Appendix} which closes this paper. 
\end{nn}

\begin{nn}{\bf Terminology and notation.}\newline
\label{NOTATION NN}
$\bullet$ By a smooth bundle we will understand a smooth submersion $p\colon E\to B$ where 
$E$ and $B$ are smooth compact manifolds. By the Ehresmann fibration theorem \cite{Ehresmann}
$p$ is then a locally trivial fiber bundle with fiber $F=p^{-1}(b)$ for $b\in B$, and with the group of 
diffeomorphisms of $F$ as the structure group. \newline
$\bullet$ All chain complexes and homology groups in this paper are taken with coefficients in $\R$, 
the real numbers.  \newline
$\bullet$ Let $f_{0}, f_{1}\colon X\to Y$ be maps of topological spaces, and let 
$h_{0}, h_{1}\colon X\times I \to Y$ be two homotopies between $f_{0}$ and $f_{1}$. 
By a homotopy of homotopies we will understand in this situation a map 
$$H\colon X\times I\times I\to Y$$
such that 
$$\left\{ \begin{array}{lcll}
H|_{X\times \{i\}\times \{t\}}& = & f_{i} & \text{for } i=0,1,\text{and } t\in I \\
H|_{X\times I\times \{j\}}& =& h_{i} & \text{for } j=0,1\\
\end{array}\right. 
$$

\end{nn}

\section{Igusa's axioms}
\label{AXIOMS SEC}
In this section we briefly summarize the main results of the work of Igusa on axioms of higher torsion. 
We refer to \cite{IgusaAx} for details. 

\begin{definition}
\label{UNIPOTENT DEF}
Let $p\colon E\to B$ be a smooth bundle such  that $B$ is a connected manifold 
with a basepoint $b_{0}$. Let $F$ be the fiber of $p$ over $b_{0}$. The bundle $p$ is
unipotent if  $H_{\ast}(F)$ admits a filtration by graded $\pi_{1}B$-submodules 
$$0=V_{0}(F)\subseteq V_{1}(F)\subseteq \dots V_{k}(F)=
H_{\ast}(F)$$
such that $\pi_{1}B$ acts trivially on the quotients $V_{i}(F)/V_{i-1}(F)$.
\end{definition}

\begin{definition}
\label{CHAR CLASS DEF}
A characteristic class $t$ of unipotent bundles in degree $k$ is an assignment which associates to every  
unipotent bundle $p\colon E\to B$ a cohomology class $t(p)\in H^{k}(B)$ in such way that for any 
smooth map $f\colon B'\to B$ we have 
$$t(p')=f^{\ast}t(p)$$
where $p'\colon f^{\ast}E\to B'$ is the bundle induced from $p$.
\end{definition} 

\begin{nn}
\label{VERTICAL BOUNDARY NN}
If $p\colon E\to B$ is a smooth bundle whose fibers are manifolds with a boundary then restricting 
$p$ to the union of boundaries of fibers of $p$ we obtain a new smooth bundle 
$$\partial^{v}p\colon \partial^{v}E \to B$$ 
which we call the vertical boundary of $p$. By \cite[Prop. 2.1]{IgusaAx} if the bundle 
$p$ is unipotent then so is $\partial^{v}p$.
\end{nn}

\begin{definition}
\label{SPLIT BUNDLE DEF}
Let $p\colon E\to B$ be a smooth bundle with closed fibers. We will say that $p$ admits a splitting 
if there exist smooth subbundles of $p$
$$p_{i}\colon E_{i}\to B,\ \   i=0, 1, 2$$
such that $p_{0}$ is the vertical boundary of both $p_{1}$ and $p_{2}$, and the bundle 
$p$ is given by 
$$p=p_{1}\cup_{p_{0}}p_{2}\colon E_{1}\cup_{E_{0}}E_{2}\to B$$ 
The splitting of $p$ is unipotent if $p_{1}$ and $p_{2}$ are unipotent bundles.
\end{definition}

If $p$ admits a unipotent splitting  $p=p_{1}\cup_{p_{0}} p_{2}$ then by 
(\ref{VERTICAL BOUNDARY NN}) the bundle $p_{0}$ is unipotent.
Using this fact and the Mayer-Vietoris sequence for the homology of the fiber of $p$ one can 
see that in such case $p$ is a unipotent bundle as well. This motivates the following

\begin{definition}[\bf Additivity Axiom \hskip -2pt
\footnote{In \cite{IgusaAx} Igusa formulates axioms for higher torsion of unipotent bundles with 
{\it closed} fibers. In effect his additivity axiom comes in a slightly different (though equivalent) 
form to the one given here.}]
\label{ADD AX DEF}
A characteristic class of unipotent bundles $t$ satisfies the additivity axiom
if for any bundle $p$ with a unipotent splitting $p=p_{1}\cup_{p_{0}}p_{2}$
we have
$$t(p)=t(p_{1})+t(p_{2})-t(p_{0})$$
\end{definition}

\begin{definition}[\bf Transfer Axiom]
\label{TRANSFER AX DEF}
Let $p\colon E\to B$ be a unipotent bundle, and let $\xi$ be an $(n+1)$-dimensional oriented 
vector bundle over $E$ with the associated sphere bundle $q\colon S^{n}(\xi)\to E$. 
Let $F$ be the fiber of $q$. We will say that a characteristic 
class $t$ satisfies the transfer axiom if for any $p$, $q$ as above we have 
$$t(pq)=\chi(F)t(p)+tr^{E}_{B}(t(q))$$
where $\chi(F)\in \Z$ is the Euler characteristic of $F$ and $tr^{E}_{B}\colon H^{\ast}(E)\to H^{\ast}(B)$ is the 
Becker-Gottlieb transfer of $p$. 
\end{definition}

The transfer axiom relies on the fact that for bundles $p$, $q$ as in (\ref{TRANSFER AX DEF}) 
the bundle $pq\colon S^{n}(\xi)\to B$ is unipotent. This follows from \cite[Prop. 2.1]{IgusaAx}.

\begin{definition}
\label{HIGHER TOR DEF}
A  characteristic class of unipotent bundles  is a higher torsion invariant if
it satisfies the additivity and transfer axioms. 
\end{definition}

We can now state the main result of \cite{IgusaAx}.

\begin{theorem}
\label{IGUSA MAIN THM}
For any $k>0$ the collection of higher torsion invariants in degree $4k$ has the structure of  
a $2$-dimensional real vector space. 
\end{theorem}

Igusa shows that the vector space of higher torsion invariants in the degree $4k$
is spanned by ``even'' and ``odd'' parts
of the Igusa-Klein torsion $t^{\it IK}_{4k}$. A closer relationship between $t^{\it IK}_{4k}$ and other 
higher torsion invariants can be obtained by restricting attention to the class of exotic invariants.

\begin{definition}[{\cite[p. 185]{IgusaAx}}] 
\label{EXOTIC DEF}
A characteristic class of unipotent bundles $t$ is exotic if for any unipotent bundle 
$p\colon E\to B$ and for any oriented linear disc bundle $q\colon D\to E$ we have 
$$t(pq)=t(p)$$
\end{definition}

\noindent Results of \cite{IgusaAx} imply that exotic higher torsion invariants form 
a $1$-dimen\-sio\-nal subspace in the vector space of higher torsion invariants. 
More precisely we have

\begin{theorem}[{\cite[Thm 9.13]{IgusaAx}}]
\label{IGUSA EXOTIC THM}
If $t$ is an exotic higher torsion invariant in degree $4k$ then there exists $\lambda\in \R$ such that 
for any unipotent bundle $p$ we have 
$$t(p)= \lambda\cdot t_{4k}^{\it IK}(p)$$
\end{theorem}


\section{Waldhausen $K$-theories}
\label{WAL SEC}

In our work on the smooth torsion of Dwyer-Weiss-Williams we will be using extensively 
the machinery of Waldhausen categories and their $K$-theories. In this section 
we review the basic notions and constructions related to this area. 
This material is standard and can be found in \cite{Wal}. 
Our goal here is to provide a concise summary of its aspects relevant to this paper  and 
to fix the notation.

\begin{nn}{\bf Waldhausen categories.} 
\label{WAL CATS NN}
By a Waldhausen category we will understand a category $\CC$ equipped with subcategories 
of cofibrations and weak equivalences satisfying the conditions of  \cite[Def. 1.2]{Wal}. 
We can turn $\CC$ into a simplicial category in two ways:
\begin{itemize}
\item[--] $\SS \CC$ is the simplicial category obtained by applying to $\CC$ the $\SSc$-construction 
of Waldhausen \cite[1.3]{Wal}; 
\item[--] $\TT \CC$ is obtained by applying to $\CC$ Thomason's variant of the 
$\SSc$-construction \cite[p. 343]{Wal}
\end{itemize}
The objects of the category  $\SS[n]\CC$ which appears in the $n$-th simplicial dimension of $\SS\CC$
are sequences of cofibrations in $\CC$:
$$c_{1} \ratail c_{2} \ratail \dots \ratail c_{n} $$
together with implicitly present quotient data. 
Morphisms in $\SS[n]\CC$ are commutative diagrams
\begin{equation*}
\label{S MORPHISM EQ}
\xymatrix{
c_{1}\  \ar@{>->}[r] \ar[d] & c_{2}\ \ar@{>->}[r] \ar[d] &\ \  \dots \ \  \ \ar@{>->}[r] & c_{n}\ar[d] \\
c'_{1}\ \ar@{>->}[r]  & c'_{2}\ \ar@{>->}[r]  &\ \  \dots \ \ \ \ar@{>->}[r]& c'_{n} \\
}
\end{equation*}
with the vertical arrows given by weak equivalences. We also set $\SS[0]\CC =\{\ast \}$. 
Notice  that $\SS[1]\CC$ is just the subcategory of weak equivalences of $\CC$. 

The category $\TT[n]\CC$ appearing in the $n$-simplicial dimension of $\TT\CC$ has as its objects
sequences of cofibrations 
$$c_{0}\ratail c_{1} \ratail c_{2} \ratail \dots \ratail c_{n} $$
Morphisms in $\TT[n]\CC$ are defined  similarly as in $\SS[n]\CC$, but the requirement that 
vertical morphisms are weak equivalences is replaced by the condition that for every  $i\geq j$ the map 
$$c'_{i}\cup_{c_{i}}c_{j}\to c'_{j}$$ 
is a weak equivalence. 
\end{nn}

\begin{nn}
\label{K-THEORY NN}
Consider the spaces $\Omega|\SS \CC|$ and $\Omega(|\TT\CC|/|\TT[0]\CC|)$. 
By \cite{Wal} these are weakly equivalent 
infinite loop spaces, representing different combinatorial models of 
the $K$-theory of the Waldhausen category $\CC$. 
\end{nn}

\begin{nn}
\label{ST MAP NN}
We have the standard maps 
$$|\SS[1]\CC|\times \Delta^{1}\to |\SS\CC| \ \ \text{and}\ \   |\TT[1]\CC| \times \Delta^{1}\to |\TT\CC|$$
By adjunction they induce canonical maps 
$$k\colon |\SS[1]\CC|\to \Omega|\SS\CC| \ \ \text{and}\ \  k\colon |\TT[1]\CC|\to \Omega(|\TT\CC|/|\TT[0]\CC|)$$
As a result given any small category $\DD$ and a functor $F\colon \DD\to \SS[1]\CC$ we obtain a map 
$$|\DD|\overset{|F|}{\lra} |\SS[1]\CC|\overset{k}{\lra} K(\R)$$
Analogously, any functor $F\colon \DD\to \TT[1]\CC$ induces a map 
$$|\DD|\to \Omega(|\TT\CC|/|\TT[0]\CC|)$$

Notice that using the map $k$ we can identify objects of $\CC$ with points in the space $\Omega |\SS\CC|$.
Similarly, cofibrations in $\CC$ can be identified with points in $ \Omega(|\TT\CC|/|\TT[0]\CC|)$.
\end{nn}

 In \cite{Wal} Waldhausen defines the notion of an exact functor of Waldhausen categories.
The main property of such functors is that they preserve the $\SSc$-construction. The following 
definition gives a slightly relaxed variant of exactness.  

\begin{definition}
\label{ALMOST EX DEF}
Let $\CC$, $\DD$ be Waldhausen categories. A functor $F\colon \CC\to \DD$ is almost exact
if it preserves cofibrations and weak equivalences and if for any diagram in $\CC$ of the form
$$c' \la c \ratail c''$$
the map 
$$F(c')\cup_{F(c)} F(c'') \to F(c' \cup_{c} c'')$$ 
is a weak equivalence in $\DD$.
\end{definition}

\begin{nn} 
\label{ALMOST EX NN}
An almost exact functor $F$ induces a functor of simplicial categories  
$F_{\bullet}\colon \TT\CC\to \SS\DD$ where $F_{n}\colon \TT[n]\CC\to \SS[n]\DD$
is given by 
\begin{multline*}
F_{n}(c_{0}\ratail c_{1} \ratail \dots \ratail c_{n})=  
(F(c_{1})/F(c_{0})\ratail \dots \ratail F(c_{n})/F(c_{0}))
\end{multline*}
Here  $c_{i}/c_{0}:=\colim(\ast \la c_{o}\ratail c_{i})$ and $\ast\in \CC$ is the terminal object. 
As a consequence $F$ defines a map 
$$\Omega(|\TT\CC|/|\TT[0]\CC|) \to \Omega|\SS\DD|$$
If $F$ is the identity functor on $\CC$ then this map gives the weak equivalence of (\ref{K-THEORY NN}). 
\end{nn}

\begin{nn}{\bf Waldhausen's pre-additivity theorem.}
\label{WAL ADD NN}
If $c, c'$ are objects in $\CC$ then the coproduct $c\sqcup c'$ represents the sum 
$c+c'$ in the $H$-space structure on $\Omega|\SS\CC|$. Similarly, taking coproducts of 
cofibrations coincides with addition in $\Omega(|\TT\CC|/|\TT[0]\CC|)$.
Waldhausen's additivity theorem \cite[Proposition 1.3.2]{Wal} relates the $H$-space structure 
on the $K$-theory of $\CC$ with  the cofibration sequences in $\CC$. 
In this paper we will use a simplified, combinatorial formulation of this theorem 
which can be described as follows. 

For a Waldhausen category $\CC$ consider the evaluation functors
$$Ev_{i}\colon \SS[2]\CC\to \SS[1]\CC, \ \ i=1,2$$
given by  $Ev_{i}(c_{1}\ratail c_{2}):= c_{i}$. Also, let $Ev_{12}\colon \SS[2]\CC\to \SS[1]\CC$ 
be defined by $Ev_{12}(c_{1}\ratail c_{2}):= c_{2}/c_{1}$. Passing to the nerves of categories
we obtain maps 
$$|Ev_{i}|\colon |\SS[2]\CC|\to |\SS[1]\CC|, \ \ i=1, 2, 12$$
We have
\begin{theorem}[{\cite[1.3.3]{Wal}}]
\label{WAL ADD THM}
Let $k$ is the map described in (\ref{ST MAP NN}). There exists a homotopy 
$$\mho\colon |\SS[2]\CC|\times I \to \Omega|\SS\CC|$$
between the maps $k\circ |Ev_{2}|$ and $k\circ |Ev_{1}|+ k\circ |Ev_{12}|$. 
\end{theorem}

Theorem \ref{WAL ADD THM} can be equivalently formulated using the $\TTc$-construction. 
In this case the functors $Ev_{i}\colon \TT[2]\CC\to \TT[1]\CC$ are defined by
$$
Ev_{i}(c_{0}\ratail c_{1}\ratail c_{2}) :=
\begin{cases}
c_{0} \ratail c_{i} & i=1,2 \\
c_{1} \ratail c_{2} & i=12
\end{cases}
$$
We will call Theorem \ref{WAL ADD THM} Waldhausen's pre-additivity theorem. 
 
\end{nn}

\begin{nn}
In this paper we will typically use Theorem \ref{WAL ADD THM} in the following way. 
Assume that for a small category $\DD$ and a Waldhausen category $\CC$ we have a functor 
$F\colon \DD\to \SS[2]\CC$ (or $F\colon \DD\to \TT[2]\CC$). By (\ref{ST MAP NN}) 
for $i=1,2,12$ the compositions $Ev_{i}\circ F$ induce maps $f_{i}\colon |\DD|\to \Omega|\SS\CC|$ 
(or respectively $f_{i}\colon |\DD|\to \Omega(|\TT\CC|/|\TT[0]\CC|)$). Then $\mho$ defines 
a preferred homotopy 
$$f_{2} \simeq f_{1}+f_{12}$$
\end{nn}


\section{The smooth torsion of Dwyer-Weiss-Williams}
\label{DWW SEC}

Below we review the construction of the Dwyer-Weiss-Williams smooth torsion 
for unipotent bundles. Our approach is the same as that of  \cite{BDW}, although the
notation differs in some  places.  

Since the constructions described in this section are rather technical, it may be useful
to keep in mind the following rough idea which motivates them.  
For any fibration $p\colon E\to B$ with a finitely dominated fiber $F$ the action of $\pi_1 B$ 
on $H_{\ast}(F, \R)$ yields a map 
$$c_p:  B  \to  K(\R)$$
where $K(\R)$ is the infinite loop space of the algebraic $K$-theory of $\R$. 
The real cohomology of $K(\R)$ contains the Borel regulator classes. 
Pulling back these classes along $c_{p}$ we obtain characteristic classes  of
$p$, which take values in the cohomology groups $H^{4k+1}(B, \R)$. 
These classes are primary invariants of $p$.

If the action of $\pi_1 B$ on $H_{\ast}(F, \R)$ is unipotent (\ref{UNIPOTENT DEF}), 
then we get a preferred homotopy $\omega_{p}$ from $c_p$ to a constant map. 
This yields a  trivialization of our primary characteristic classes.

Also, if  $p$ is a smooth bundle, then $c_p$ has a factorization $p^!$ through 
a space homotopy equivalent to $\Omega^\infty\Sigma^\infty(S^0)$ which is rationally 
a discrete space.  Therefore the smooth structure  on $p$ determines another trivialization 
of the primary characteristic classes.

As a consequence, for $p$ which is both smooth and unipotent we have two trivializations of our
primary invariants. Taken together they  yield secondary characteristic classes 
$t^{s}_{4k}(p)\in H^{4k}(B, \R)$ which are the Dwyer-Weiss-Williams smooth torsion classes.

\begin{nn}{\bf Transfer.}
\label{TR LIN NN}
Let $p\colon E\to B$ be a smooth bundle. We will denote by $\simp$ the category of smooth 
singular simplices of $B$ \cite[4.1]{BDW} and by $|\simp|$ the  geometric realization of 
the nerve of $\simp$. We have a weak equivalence $B\simeq |\simp|$.
As in \cite[Sec. 3]{BDW}  by $\tQ(E_{+})$ we will understand 
the space obtained by applying the $\TTc$-construction (\ref{WAL CATS NN}) 
to the category of partitions of $E$. We have a weak equivalence
$$\tQ(E_{+})\simeq \Omega^{\infty}\Sigma^{\infty}E_{+}$$

Following \cite[Sec. 4]{BDW} by $p^{!}\colon |\simp|\to \tQ(E_{+})$ 
we will denote the Becker-Gottlieb transfer of the 
bundle $p$ and by $\tQ(p^{!})\colon \tQ(B_{+})\to \tQ(E_{+})$ its extension to $\tQ(B_{+})$. 

Next, let $\Rfd(E)$ be the category of homotopy finitely dominated retractive spaces over $E$ 
with maps of retractive spaces as morphisms. It is a Waldhausen category with cofibrations given 
by Serre cofibrations and weak equivalences defined as weak homotopy equivalences. Denote
$$A(E):=\Omega(|\TT \Rfd(E)|/|\TT[0]\Rfd(E)|)$$
This is the Waldhausen $A$-theory of the space $E$. We have the assembly map 
$\ta_{E}\colon \tQ(E)\to A(E)$ \cite[Sec. 3]{BDW}.

Let $\sigma\colon \Delta^{k}\to B$ be a smooth singular simplex and let 
$$\sigma^{\ast}E:=\lim (\Delta^{k}\overset{\sigma}{\lra} B \overset{p}{\longleftarrow} E)$$ 
The space $\sigma^{\ast}E\sqcup E$ is in the obvious way a retractive space over $E$. 
Consider the functor 
\begin{equation}
\label{FpA EQ}
F_{p^{A}}\colon \simp \to \TT[1]\Rfd(E)
\end{equation}
which assigns to $\sigma\in \simp$ the cofibration $F_{p^{A}}(\sigma):=(E\ratail E\sqcup \sigma^{\ast}E)$. 
By (\ref{ST MAP NN}) the functor $F_{p^{A}}$ induces a map $p^{A}\colon |\simp|\to A(E)$. 
The combinatorial constructions of the maps $p^{!}$ and $\ta_{E}$ in \cite{BDW} 
imply  that the following diagram commutes: 
\begin{equation}
\label{pA EQ}
\xymatrix{
& \tQ(E_{+})\ar[d]^{\ta_{E}} \\
|\simp|\ar[ru]^{p^{!}}\ar[r]_{p^{A}} & A(E) \\
}
\end{equation}
\end{nn}

\begin{remark} We note that commutativity of the diagram (\ref{pA EQ}) depends in an essential 
way on the smooth structure of the bundle $p$. Indeed, while all maps appearing in this diagram 
exist for any fibration $p\colon E\to B$  with finitely dominated fibers, Dwyer has given examples
of fibrations where the map $\ta_{E}\circ p^{!}$ is not homotopic to $p^{A}$ 
(see proof of Theorem F in \cite{KW}).
\end{remark}

\begin{nn}{\bf Linearization.}
\label{LIN NN}
By $\Chfd(\R)$ we will denote the category of homotopy finitely dominated chain complexes of 
$\R$-vector spaces. This is a Waldhausen category with degreewise monomorphisms as cofibrations 
and quasi-isomorphisms as weak equivalences. Applying the $\SSc$-construction 
(\ref{WAL CATS NN}) we obtain the space 
$$K(\R):=\Omega|\SS\Chfd(\R)|$$
which has the homotopy type of  the infinite loop space underlying the algebraic $K$-theory 
spectrum of $\R$.

Consider the linearization functor 
\begin{equation}
\label{LAMBDA FUNCTOR EQ}
\Lambda^{A}_{E}\colon \Rfd(E)\to \Chfd(\R)
\end{equation}
which assigns to a retractive space $X\in \Rfd(E)$ the singular chain complex $C_{\ast}(X)$.
This functor is almost exact (\ref{ALMOST EX DEF}) so it induces a map 
$$\lambda^{A}_{E}\colon A(E)\to K(\R)$$ 
The composition
$$\lambda_{E}:=\lambda^{A}_{E}\circ \ta_{E}\colon \tQ(E_{+})\to K(\R)$$
is the linearization map of $E$.

Next, let 
$$C_{p}\colon \simp \to \SS[1]\Chfd(\R)$$ 
be the functor given by $C_{p}(\sigma)=C_{\ast}(\sigma^{\ast}E)$. 
It induces a map $c_{p}\colon |\simp|\to K(\R)$. 
Notice that we have a canonical isomorphism of chain complexes 
$$C_{p}(\sigma)\cong \Lambda^{A}_{E, 1}\circ F_{p^{A}}(\sigma)$$
where 
$$\Lambda^{A}_{E, 1}\colon \TT[1]\Rfd(E)\to \SS[1]\Chfd(\R)$$ 
is the functor induced by 
$\Lambda^{A}_{E}$ (see \ref{ALMOST EX NN}).
This shows that the following diagram commutes up to a preferred homotopy:
$$
\xymatrix{
& A(E) \ar[d]^{\lambda^{A}_{E}} \\
|\simp| \ar[ur]^{p^{A}} \ar[r]_{c_{p}} & K(\R) \\
}
$$
Combining this with the diagram (\ref{pA EQ}) we obtain a diagram
$$
\xymatrix{
& \tQ(E_{+})\ar[d]^{\ta_{E}} \ar@(r, r)[dd]^{\lambda_{E}} \\
& A(E) \ar[d]^{\lambda^{A}_{E}} \\
|\simp|\ar[r]_{c_{p}} \ar[ru]^{p^{A}}\ar@(u, l)[ruu]^{p^{!}} & K(\R) \\
}
$$
which is homotopy commutative (via preferred homotopies). 
\end{nn}

\begin{nn}{\bf Smooth torsion.}
In the next definition we use the fact (\ref{ST MAP NN}) that  chain complexes 
in $\Chfd(\R)$ can be identified with points in the space $K(\R)$.

\begin{definition}
\label{WH DEF}
For a chain complex $C\in \Chfd(\R)$ the Whitehead space $\wh(E)_{C}$ is the homotopy fiber 
$$\wh(E)_{C}:=\hofib(\lambda_{E}\colon \tQ(E_{+})\to K(\R))_{C}$$
\end{definition}
\noindent We will denote by $\wh(E)$ the space $\wh(E)_{0}$ 
where $0\in \Chfd(\R)$ is the zero chain complex.
Since $\lambda_{E}$ is a map of infinite loop spaces $\wh(E)$ has a natural infinite loop space structure. 

Assume now that $p\colon E\to B$ is a unipotent bundle (\ref{UNIPOTENT DEF}) with a basepoint 
$b_{0}\in B$ and let $F$ be the fiber of $p$ over $b_{0}$. Consider the graded vector space $H_{\ast}(F)$ 
as a chain complex with trivial differentials. 
By \cite[Thm. 6.7]{BDW} we have a preferred homotopy 
$$\omega_{p}\colon |\simp|\times I \to K(\R)$$
such that $\omega_{p}|_{|\simp|\times \{0\}}=c_{p}$ and $\omega_{p}|_{|\simp|\times \{1\}}$ is the constant 
map $\bast_{H_{\ast}(F)}$, which sends $|\simp|$ to the point $H_{\ast}(F)\in K(\R)$. 
We will call $\omega_{p}$ the algebraic contraction of $p$.

\begin{definition}[{\cite[6.9, 6.10]{BDW}}]
\label{STOR DEF}
The unreduced smooth torsion of a unipotent bundle $p\colon E\to B$ is the map 
$$\ttaus(p)\colon |\simp|\to \wh(E)_{H_{\ast}(F)}$$
which is the lift of the Becker-Gottlieb transfer $p^{!}$ determined by the algebraic contraction $\omega_{p}$.
The (reduced) smooth torsion of $p$ is the map 
$$\taus(p)\colon |\simp|\to \wh(E)$$
obtained by shifting $\ttaus(p)$ to $\wh(E)$.
\end{definition}

\noindent More precisely, consider the map ${\bar p}^{!}\colon |\simp| \to \tQ(E_{+})$
obtained by subtracting from $p^{!}$ the constant map which sends $|\simp|$ 
into the point $p^{!}(b_{0})\in \tQ(E_{+})$. Also, 
let ${\bar \omega}_{p}\colon |\simp|\times I\to K(\R)$ be the homotopy which for every $t\in I$ is given by 
the substracing from the map  $\omega_{p}|_{|\simp|\times \{t \}}$ the constant map sending $|\simp|$ to 
$\omega_{p}(b_{0}, t)$. Then ${\bar\omega}_{p}$ is a homotopy between  $\lambda_{E}\circ\bar{p}^{!}$
and the constant map sending $|\simp|$ to $0\in K(\R)$. We will call ${\bar\omega}_{p}$ the reduced 
algebraic contraction of $p$. The reduced torsion $\tau^{s}(p)$ is the lift 
of ${\bar p}^{!}$ determined by ${\bar\omega}_{p}$.
\end{nn}

\begin{remark}
\label{REDUCTION REM} 
The following observation will be useful later on. Assume that for a map
$$f\colon |\simp|\to \tQ(E_{+})$$ 
and a chain complex $C\in \Chfd(\R)$
we have a homotopy 
$$\omega\colon |\simp|\times I\to K(\R)$$ 
between $\lambda_{E} f$ and the constant map $\bast_{C}$. The pair $(f, \omega)$ 
defines a map $\varphi\colon |\simp|\to \wh(E)_{C}$. Reducing $f$ and $\omega$ 
the same way which we used above to obtain $\bar p^{!}$ and $\bar \omega_{p}$ we get 
a pair $(\bar f, \bar \omega)$ which determines a map $\bar \varphi\colon |\simp|\to \wh(E)$.
We will call $\bar\varphi$ the reduction of $\varphi$.

Let $\varphi'\colon |\simp|\to \wh(E)_{C'}$ be another map defined by a pair 
$(f', \omega')$ and let $\bar \varphi'$ be the reduction of $\varphi'$.  
Notice that $\bar\varphi$ and $\bar\varphi'$ are homotopic maps if the following
conditions hold:
\begin{itemize}
\item[-] there exists a homotopy $h\colon |\simp|\times I\to \tQ(E_{+})$ between $f$ and $f'$;
\item[-] there exists a path $\gamma$ in $K(\R)$ joining the points $C$ and $C'$;
\item[-] there exists a homotopy of homotopies (\ref{NOTATION NN})
between the concatenation of $\lambda_{E} h$ with $\omega'$, and the concatenation 
of $\omega$ with $\gamma$ (we interpret here $\gamma$ as
a homotopy $\gamma\colon |\simp|\times I\to K(\R)$ via constant maps):
$$
\xymatrix{
\lambda_{E}f \ar[r]^{\lambda_{E} h}\ar[d]_{\omega} & \lambda_{E}f'\ar[d]^{\omega'} \\
\bast_{C}\ar[r]_{\gamma} & \bast_{C'} \\
}
$$
In the above diagram vertices represent maps $|\simp|\to K(\R)$ and edges stand for homotopies
of such maps.  
\end{itemize}

\end{remark}

\begin{nn}{\bf The algebraic contraction.}
\label{ALG CONTR NN}
Since our arguments later in this paper will rely on the specifics of the construction of the algebraic contraction
$\omega_{p}$ we will now review the main steps of this construction. 

As before, for a unipotent bundle $p\colon E\to B$ let $b_{0}$ be the basepoint 
of $B$ and let $F=p^{-1}(b_{0})$. 
The homotopy $\omega_{p}$ is obtained by concatenating the three following homotopies:

\begin{list}{}{\itemsep 3mm \leftmargin 0mm \itemindent 2mm}
\item[$\bullet$ \bf The homotopy $\bm{\omega^{(1)}_{p}}$.] 
Let 
$$H_{p}\colon \simp\to \SS[1]\Chfd(\R)$$
denote the functor which assigns to a singular simplex $\sigma\in \simp$ the homology chain complex 
$H_{\ast}(\sigma^{\ast}E)$. It induces a map 
$$h_{p}\colon |\simp|\to K(\R)$$
The homotopy $\omega^{(1)}_{p}$ is the homotopy between the maps $c_{p}$ and $h_{p}$ described in 
Remark \ref{OMEGA1 REM}.  

\item[ $\bullet$ \bf  The homotopy $\bm{\omega^{(2)}_{p}}$.] 
For $\sigma\in\simp$ let $\sigma(0)\in B$ denote the zeroth vertex
of $\sigma$ and let $F_{\sigma(0)}=p^{-1}(\sigma(0))$. We have a functor
$$H^{0}_{p}\colon \simp \to \SS[1]\Chfd(\R)$$
such that $H^{0}_{p}(\sigma)=H_{\ast}(F_{\sigma(0)})$. 
Let  $h^{0}_{p}\colon |\simp|\to K(\R)$ be the map induced by the functor $H^{0}_{p}$. 
The isomorphisms of chain complexes 
$$H_{\ast}(\sigma^{\ast}E)\overset{\cong}{\lra} H_{\ast}(F_{\sigma(0)})$$
define a natural transformations of the functors $H_{p}$ and $H^{0}_{p}$, and thus 
induce a homotopy $\omega^{(2)}_{p}$ between the maps $h_{p}$ and $h^{0}_{p}$.

\item[$\bullet$ \bf The homotopy $\bm{\omega^{(3)}_{p}}$.] 
As before let 
$${\bast}_{H_{\ast}(F)}\colon |\simp|\to K(\R)$$ 
denote the constant map sending the space $|\simp|$ to the point of $K(\R)$ 
represented by the chain complex $H_{\ast}(F)$. If $p$ is a bundle such that the group 
$\pi_{1}B$ acts trivially on $H_{\ast}(F)$ then for every $\sigma\in \simp$ we have 
a canonical isomorphism 
$$H_{\ast}(F_{\sigma(0)})\overset{\cong}{\lra} H_{\ast}(F)$$
These isomorphisms define a homotopy $\omega^{(3)}_{p}$ between
 $h^{0}_{p}$ and ${\bast}_{H_{\ast}(F)}$.

This construction can be generalized to the case when $p$ is an arbitrary unipotent bundle.
 We have then canonical isomorphisms of  quotients of  the filtrations  of $H_{\ast}(F_{\sigma(0)})$ 
 and $H_{\ast}(F)$ given by Definition \ref{UNIPOTENT DEF}.
The homotopy $\omega^{(3)}_{p}$ is obtained using this 
fact and Waldhausen's pre-additivity theorem (see \cite[Proof of Thm 6.7]{BDW}). We note here that
that the homotopy class of $\omega^{(3)}_{p}$ does not depend on the choice of the filtration
of $H_{\ast}(F)$.

\end{list}
\end{nn}

\begin{nn}{\bf Smooth cohomological torsion.}
\label{COH STOR NN}
Recall (Sec. \ref{AXIOMS SEC}) that in the axiomatic setting of Igusa higher torsion is defined as an invariant 
taking values in the cohomology groups of the base of the bundle. As a consequence in order to verify that 
Igusa's axioms hold for the smooth torsion one needs first to reduce $\taus$ 
to a cohomological invariant. This is accomplished as follows (cf. \cite[Sec. 7]{BDW}). 
Let $p\colon E\to B$ be a unipotent bundle and let $\iota_{E}\colon \tQ(E_{+})\to \tQ(S^{0})$ 
be the augmentation map. Consider the diagram
\begin{equation}
\label{COH DIAG EQ}
\xymatrix{
  & \wh(E)\ar[d]  \ar@{-->}[r]^{{\bar\iota}_{E}}& \wh(\ast)\ar[d]\\
|\simp|\ar[ru]^{\taus(p)} \ar[r]^{p^{!}} & \tQ(E_{+})\ar[d]_{\lambda_{E}} \ar[r]^{\iota_{E}} & \tQ(S^{0})
\ar[d]^{\lambda_{\ast}} \\
 & K({\mathbb R}) \ar@{=}[r]& K(\R)\\
}
\end{equation}
The lower square commutes up to a preferred homotopy, and so $\iota_{E}$ induces a map 
${\bar\iota}_{E}\colon \wh(E)\to \wh(\ast)$.  We have weak equivalences
\begin{equation}
\label{RATIONAL WE EQ}
\wh(\ast)_{\Q}\simeq \Omega K(\R)_{\Q} \simeq K_{1}(\R)\times \prod_{k>1}K(\R, 4k)
\end{equation}
where $\wh(\ast)_{\Q}$, $\Omega K(\R)_{\Q}$ denote rationalizations of $\wh(\ast)$ and $\Omega K(\R)$
respectively. The first weak equivalence is a consequence of the fact that 
$\tQ(S^{0})$ is rationally a discrete space and that $\lambda_{\ast}$ induces an isomorphism on 
the level of $\pi_{0}$.
The second weak equivalence in (\ref{RATIONAL WE EQ}) is given by the Borel regulator maps \cite{Borel}
\hskip -2pt \footnote{We rectify here a mistake in Section 7 of \cite{BDW} where it was incorrectly stated that 
$\wh(\ast)_{\Q}$ is a connected space.}. 

Consider the map 
\begin{equation}
\label{IOTA TAUS EQ}
{\bar \iota_{E}}\circ\taus(p)\colon |\simp|\to \wh(\ast)_{\Q}
\end{equation}
We claim that the image of this map always lies in the connected component of the identity element of 
$\wh(\ast)_{\Q}$. This is obvious if $B=\ast$. The general case follows from this fact and naturality
of the map (\ref{IOTA TAUS EQ}): if $p\colon E\to B$ is a unipotent bundle, $f\colon B'\to B$ 
is a smooth map and  $p'\colon f^{\ast}E\to B'$ is the bundle induced from $p$ then we have a homotopy 
 $${\bar \iota_{f^{\ast}E}}\circ\taus(p')\simeq {\bar \iota_{E}}\circ\taus(p)\circ |f|$$
 where $|f|\colon |\simp[B']|\to |\simp|$ is the map induced by $f$.
The proof of this property is the same as that of \cite[Proposition 7.3]{BDW}.

It follows that for a unipotent bundle $p$ the target of the map ${\bar \iota_{E}}\circ\taus(p)$ can be 
identified with $\prod_{k>1} K(\R, 4k)$. This gives rise the following 
\begin{definition}
\label{COH STOR DEF} 
For a unipotent bundle $p\colon E\to B$ the smooth cohomological torsion of $p$ is the cohomology class 
$\ts(p)\in \bigoplus_{k>1}H^{4k}(B; \R)$ represented by the map 
$${\bar \iota_{E}}\circ\taus(p)\colon |\simp|\to \wh(\ast)_{\Q}$$
\end{definition}

\noindent We use here the canonical identification of cohomology groups of the spaces $B$ and $|\simp|$.
\end{nn}

We are now ready to  state the main result of this paper.

\begin{theorem}
\label{MAIN THM}
For $k>0$ and a unipotent bundle  $p\colon E\to B$ let 
$$t^{s}_{4k}(p)\in H^{4k}(B; \R)$$
be the  degree $4k$ component of $t^{s}(p)$. 
The invariant  $t^{s}_{4k}$ is a non-trivial exotic higher torsion invariant of unipotent bundles 
in degree $4k$.
\end{theorem}

Combining this with Igusa's Theorem \ref{IGUSA EXOTIC THM} we obtain 

\begin{theorem}
\label{MAIN2 THM}
For each $k>0$ there exists $0\neq \lambda_{4k}\in \R$ such that for every unipotent bundle 
$p\colon E\to B$ we have 
$$t^{s}_{4k}(p)=\lambda_{4k}t^{\it IK}_{4k}(p)$$
where $t^{\it IK}_{4k}$ is the Igusa-Klein higher torsion invariant in the degree $4k$.
\end{theorem}

\begin{remark}
 It would be interesting to know the exact value of the proportionality constant $\lambda_{4k}$.
 Such computation seems within reach, and one should be able  accomplish it by a careful analysis 
 of the map $G/O \to \Omega \wh[\rm diff](\ast)$ constructed by Waldhausen in \cite{WalM1}. 
 We do not attempt it in this paper.
 \end{remark}

The remainder of this paper is devoted to the proof of Theorem \ref{MAIN THM}. 
The fact the cohomological smooth torsion defines  characteristic classes of unipotent 
bundles was proved in \cite[Theorem 7.3]{BDW}. Also, directly from the constructions 
in \cite[Section 7]{BDW} it follows that these characteristic classes are exotic
\hskip -2pt\footnote{This fact was also observed by Igusa \cite[p. 185]{IgusaAx}.}. It remains to show 
that $t^{s}$ is a higher torsion invariant i.e. that it satisfies the additivity (\ref{ADD AX DEF}) 
and transfer (\ref{TRANSFER AX DEF}) axioms. We verify the additivity formula for $t^{s}$ 
in the next section (see Corollary \ref {ADD COR}). The transfer axiom is the subject of 
Corollary \ref{TRANSFER COR}. Finally, in Section \ref{NONTRIVIALITY SEC} we show 
that for every $k>0$ there exists a unipotent bundle $p$ such that $t^{s}_{4k}(p)\neq 0$. 
This shows that the characteristic classes $t^{s}_{4k}$ are non-trivial invariants.

\section{Additivity of  the smooth torsion}
\label{ADD TORSION SEC}
As we have indicated above our goal in this section  is to verify that the  smooth torsion 
satisfies an analog of the additivity axiom of Igusa (\ref{ADD AX DEF}):

\begin{theorem}
\label{ADD THM}
Let $p\colon E\to B$ be a bundle with a unipotent splitting (\ref{SPLIT BUNDLE DEF})
$$p=p_{1}\cup_{p_{0}}p_{2}$$ 
where $p_{i}\colon E_{i}\to B$. For $i=0, 1, 2$ let $j_{i\ast}\colon \wh(E_{i})\to \wh(E)$
be the map induced by the inclusion $j_{i}\colon E_{i}\hra E$.
There exists a homotopy 
$$\taus(p)  \simeq j_{1\ast}\taus(p_{1})+j_{2\ast}\taus(p_{2}) + \bar g^{\wh}$$
where $\bar g^{\wh}$ is a map representing the homotopy type of $-j_{0\ast}\taus(p_{0})$
in the $H$-space structure of $\wh(E)$. 
\end{theorem}

From Theorem \ref{ADD THM} we immediately obtain 

\begin{corollary}
\label{ADD COR}
The smooth cohomological torsion $\ts$ satisfies the additivity axiom (\ref{ADD AX DEF}). 
\end{corollary}

As we indicated in Section \ref{DWW SEC} the smooth torsion $\taus(p)$ of a unipotent bundle 
$p\colon E\to B$ consists of two main ingredients: the Becker-Gottlieb transfer 
$p^{!}\colon |\simp|\to \tQ(E_{+})$ and the algebraic contraction $\omega_{p}\colon |\simp|\times I\to K(\R)$.
Our first aim is to describe the additive property of the transfer map. Recall the commutative 
diagram (\ref{pA EQ}). We have

\begin{theorem}
\label{ADD TRANSFER THM}
Let $p\colon E\to B$ be a smooth bundle with a unipotent splitting
$$p=p_{1}\cup_{p_{0}}p_{2}$$ 
where $p_{i}\colon E_{i}\to B$.  Let $j_{i\ast}\colon \tQ(E_{i+})\to \tQ(E_{+})$ and 
$j_{i\ast}\colon A(E_{i})\to A(E)$ be the maps induced by the inclusions $j_{i}\colon E_{i}\hra E$.

\vskip 2mm
\begin{list}{}{\itemsep 3mm \leftmargin 0mm \itemindent 2mm}
\item[{\it (i)}] There exits a preferred homotopy 
$$\gamma^{A} \colon |\simp|\times I \to A(E)$$
between the maps $p^{A}$ and $j_{1\ast}p^{A}_{1}+j_{2\ast}p^{A}_{1} + g^{A}$, 
where $g^{A}$ is a map representing the homotopy type of $-j_{0\ast}p^{A}_{0}$.
\item[{\it (ii)}] There exits a preferred homotopy 
$$\gamma^{Q} \colon |\simp|\times I \to \tQ(E_{+})$$
between the maps $p^{!}$ and $j_{1\ast}p^{!}_{1}+j_{2\ast}p^{!}_{1} + g^{Q}$
where $g^{Q}$ is a map representing the homotopy type of $-j_{0\ast}p^{!}_{0}$.
Moreover, we have $\ta_{E}g^{Q}=g^{A}$ and  $\ta_{E}\circ \gamma^{Q}=\gamma^{A}$.
\end{list}
\end{theorem}

\begin{remark}
Existence of the homotopy $\gamma^{Q}$ is well known; it is in fact one of the properties characterizing the
Becker-Gottlieb transfer map given in \cite{Becker-Schultz}. 
The proof of  Theorem \ref{ADD THM} will require, however, 
a combinatorial description of $\gamma^{A}$ and $\gamma^{Q}$ given in the proof of Theorem 
\ref{ADD TRANSFER THM}.
\end{remark}

\begin{proof}[Proof of Theorem \ref{ADD TRANSFER THM}]
{\em Part (i).} Our construction of the homotopy $\gamma^{A}$ will parallel the one given in \cite[Section 3]{BD}.
Let $b\colon E_{0}\times [1, -1]\to E$ be a fiberwise bicollar neighborhood of $E_{0}$ in $E$. Thus,
$b$ is a smooth embedding such that $b(E_{0}\times \{0\})=E_{0}$, $b(E_{0}\times [-1, 0])\subseteq E_{1}$,
$b(E_{0}\times [0, 1])\subseteq E_{2}$, and such that we have a commutative diagram 
$$\xymatrix{
E_{0}\times [-1, 1] \ar[rr]^{b}\ar[rd]_{p_{0}\circ \ { pr_{1}}} & & E \ar[ld]^{p}\\
& B & \\
}$$ 
where $pr_{1}\colon E_{0}\times [-1, 1]\to E_{0}$ is the projection on the first factor. 
Define 
$$E_{1}' := E_{1} - b(E_{0}\times (-1, 0]) \text{\  \ and \ } E'_{2}:=E_{2} - b(E_{0}\times [0, 1))$$
Restricting $p$ to $E_{i}'$ we obtain smooth bundles 
$$q_{i}\colon E'_{i}\to B, \ \ i=1,2$$
Let $j'_{i}\colon E'_{i}\hra E$ be the inclusion map. The bundle $q_{i}$ is fiberwise diffeomorphic 
to $p_{i}$ and we have a homotopy 
$$j'_{i\ast}q_{i}^{A}\simeq j_{i\ast}p_{i}^{A}$$
In view of this it suffices to construct a homotopy $\tilde\gamma^{A}$ between the maps 
$p^{A}$ and $j'_{1\ast}q^{A}_{1}+j'_{2\ast}q^{A}_{2} + g^{A}$ for an appropriate choice of $g^{A}$.

Recall  that the map $p^{A}$ is induced by the functor $F_{p^{A}}\colon \simp \to \TT[1]\Rfd(E)$
(\ref{FpA EQ}). For $i=1, 2$ let 
$$\oF_{q^{A}_{i}}\colon \simp \to \TT[1]\Rfd(E)$$
be the functor which assigns to a simplex $\sigma\in \simp$ the cofibration 
$$\oF_{q^{A}_{i}}(\sigma):=(E\ratail E\sqcup \sigma^{\ast}E'_{i})$$
where $ \sigma^{\ast}E'_{i}$ is defined as in (\ref{TR LIN NN}). 
The maps $j'_{i\ast}q^{A}_{i}\colon |\simp|\to A(E)$ are obtained from the functors $\oF_{q_{i}^{A}}$ 
using (\ref{ST MAP NN}). Likewise, the map $j_{0\ast}p_{0}^{A}$ comes from the functor 
$\oF_{p_{0}^{A}}\colon \simp \to \TT[1]\Rfd(E)$ given by 
$$\oF_{p_{0}^{A}}(\sigma):=(E\hra E\sqcup \sigma^{\ast} E_{0})$$
It follows that the map $j'_{1\ast}q^{A}_{1}+j'_{2\ast}q^{A}_{2}$ is represented by the functor 
$$\oF'_{p_{1}^{A}}\sqcup \oF'_{p_{2}^{A}}\colon \simp\to \TT[1]\Rfd(E)$$
where
$$\oF'_{p_{1}^{A}}\sqcup \oF'_{p_{2}^{A}}(\sigma)=
(E\ratail E\sqcup \sigma^{\ast}E'_{1}\sqcup \sigma^{\ast} E'_{2})$$
Notice  that for every  $\sigma\in \simp$ we have a sequence of cofibrations
\begin{equation}
\label{GAMMA COF SEQ EQ}
E\ratail  E\sqcup (\sigma^{\ast}E'_{1}\sqcup \sigma^{\ast} E'_{2})\ratail E\sqcup \sigma^{\ast}E
\end{equation}
Consider the functor 
\begin{equation}
\label{GAMMA FUNCT EQ}
\Gamma^{A}\colon \simp \to \TT[2]\Rfd(E)
\end{equation}
which assigns to $\sigma\in \simp$ the cofibration sequence (\ref{GAMMA COF SEQ EQ}).
Applying Waldhausen's pre-additivity theorem (\ref{WAL ADD THM}) to $\Gamma^{A}$ 
we obtain a homotopy between $p^{A}$ and the map
$j'_{1\ast}q^{A}_{1}+j'_{2\ast}q^{A}_{1} + g^{A}$ where $g^{A}$ is the map induced by the functor
$$G^{A}\colon \simp\to \TT[1]\Rfd(E)$$ 
given by 
\begin{equation}
\label{GA EQ}
G^{A}(\sigma):=(E\sqcup (\sigma^{\ast}E'_{1}\sqcup \sigma^{\ast} E'_{2})\ratail E\sqcup \sigma^{\ast}E)
\end{equation}
It remains to check that $g^{A}$ represents the  homotopy type of  $-j_{0\ast}p_{0}^{A}$. 
This can be seen by noticing 
that the cofiber of the cofibration $G^{A}(\sigma)$ is isomorphic to the object of $\Rfd(E)$ obtained by 
applying the suspension functor \cite[1.6]{Wal} to the cofiber of $\oF_{p_{0}^{A}}(\sigma)$.

\vskip 3mm \noindent {\em Part (ii).} As we have mentioned in (\ref{TR LIN NN}) the space 
$\tQ(E_{+})$ is constructed by applying the $\TTc$-construction to the category of partitions 
of the manifold $E$.  The map $g^{Q}$ and the homotopy $\gamma^{Q}$ can be obtained 
by arguments paralleling the ones used in the proof of the part (i), working in the category of partitions 
instead of $\Rfd(E)$.

\end{proof}

\vskip 3mm
\begin{proof}[Proof of Theorem \ref{ADD THM}]
We will use the notation introduced in the proof of Theorem \ref{ADD TRANSFER THM}.
Using diffeomorphisms between the bundles $p_{i}$ and $q_{i}$, $i=1,2$ we obtain homotopies
$$j_{i\ast}\taus(p_{i})\simeq j'_{i\ast}\taus(q_{i})$$
Therefore it will suffice to show that we have a homotopy 
\begin{equation}
\label{ADD MOD EQ}
\tau^{s}(p)\simeq  j'_{1\ast}\taus(q_{1})+j'_{2\ast}\taus(q_{2}) + \bar g^{\wh}
\end{equation}

A  convenient description of the maps  $ j'_{i\ast}\taus(q_{i})$, $i=1, 2$ can be obtained as follows. 
For $i=1, 2$ we have a commutative diagram
$$
\xymatrix{
\tQ(E'_{i+})\ar[rr]^{j'_{i\ast}}\ar[dr]_{\lambda_{E'_{i}}} & & \tQ(E_{+}) \ar[ld]^{\lambda_{E}}\\
& K(\R) & \\
}
$$
It follows that 
\begin{equation}
\label{LQ EQ}
\lambda_{E}j'_{i\ast}q_{i}^{!}=\lambda_{E'_{i}}q^{!}_{i}=c_{q_{i}}
\end{equation}
Let $F'_{i}$ be the fiber of the bundle $q_{i}$. Recall (\ref{STOR DEF}) that the unreduced smooth 
torsion of $q_{i}$ is the map 
$$\tilde\tau^{s}(q_{i})\colon |\simp|\to \wh(E'_{i})_{H_{\ast}(F'_{i})}$$
determined by the transfer $q_{i}^{!}$ and the algebraic contraction $\omega_{q_{i}}$. 
From the equation (\ref{LQ EQ}) we obtain that the map 
$j_{\ast}\tilde\tau^{s}(q_{i})\colon |\simp|\to \wh(E)_{H_{\ast}(F'_{i})}$ 
is determined by the pair $(j'_{i\ast}q_{i}^{!}, \omega_{q_{i}})$. The map 
$j'_{i\ast}\taus(q_{i})$ is then just the reduction of $j'_{\ast}\tilde\tau^{s}(q_{i})$ 
(\ref{REDUCTION REM}).

The map $\bar g^{\wh}$ is constructed as follows.  Consider the functor 
$$G^{K}\colon \simp \to \SS[1]\Chfd(\R)$$ 
which assigns to $\sigma\in \simp$ the relative chain complex 
\begin{equation}
\label{GK FUNCT EQ}
G^{K}(\sigma):=C_{\ast}(\sigma^{\ast}E, \sigma^{\ast}E'_{1}\sqcup \sigma^{\ast}E'_{2})
\end{equation}
Let $g^{K}\colon |\simp|\to K(\R)$ be the map induced by $G^{K}$ as in (\ref{ST MAP NN}).
Notice that we have
$$G^{K}=\Lambda^{A}_{E, 1}\circ G^{A}$$ 
where $G^{A}$ is the functor defined by (\ref{GA EQ}) and 
$$\Lambda^{A}_{E, 1}\colon \TT[1]\Rfd(E)\to \SS[1]\Chfd(\R)$$ 
is induced by the linearization 
functor $\Lambda^{A}_{E}$ (\ref{LAMBDA FUNCTOR EQ}), cf. (\ref{ALMOST EX NN}). 
This shows that $g^{K}=\lambda^{A}_{E}\circ g^{A}$.
Since $g^{A}$ has the homotopy type of $-j_{0\ast}p_{0}^{A}$, 
$\lambda^{A}_{E}$ is a map of infinite loop spaces, and 
$\lambda^{A}_{E}j_{0\ast}p_{0}^{A}=c_{p_{0}}$,  it follows that 
$g^{K}$ represents the homotopy type of $-c_{p_{0}}$. 

Recall that $F'_{1}$, $F'_{2}$ are the fibers of $q_{1}$, $q_{2}$, respectively, and let $F$  denote the fiber 
of $p$. One can construct a homotopy $\omega_{g}$ from  $g^{K}$ to a constant map 
$\bast_{H_{\ast}(F, F'_{1}\sqcup F'_{2})}$ following the same steps which we used to obtain the algebraic 
contraction $\omega_{p}$ in (\ref{ALG CONTR NN}). More precisely, let 
$$G^{H}, G^{H_{0}}\colon \simp\to \SS[1]\Chfd(\R)$$ 
be functors given by
$$G^{H}(\sigma):= H_{\ast}(\sigma^{\ast}E, \sigma^{\ast}E'_{1}\sqcup \sigma^{\ast}E'_{2})$$
and 
$$G^{H_{0}}(\sigma):=H_{\ast}(F_{\sigma(0)}, F'_{1,\sigma(0)}\sqcup F'_{2,\sigma(0)})$$
where $F_{\sigma(0)}$, $F'_{1,\sigma(0)}$,  $F'_{2,\sigma(0)}$ are the fibers of $p$, $q_{1}$, $q_{2}$
respectively taken over the zeroth vertex $\sigma(0)$ of the singular simplex $\sigma$.
Let
$$g^{H}, g^{H_{0}}\colon |\simp|\to K(\R)$$ 
be the maps induced by $G^{H}$ and $G^{H_{0}}$.
A homotopy $\omega^{(1)}_{g}$ between $g^{K}$ and $g^{H}$ can be obtained the same way as 
the homotopy $\omega^{(1)}_{p}$  (cf.  Remark \ref{OMEGA1 REM}). 
The natural isomorphisms 
$$H_{\ast}(\sigma^{\ast}E, \sigma^{\ast}E'_{1}\sqcup \sigma^{\ast}E'_{2})\overset{\cong}{\lra}
H_{\ast}(F_{\sigma(0)}, F'_{1,\sigma(0)}\sqcup F'_{2,\sigma(0)})
$$
induce a homotopy $\omega^{(2)}_{g}$ between $g^{H}$ and $g^{H_{0}}$. Finally, 
since $\pi_{1}B$ acts unipotently on $H_{\ast}(F, F'_{1}\sqcup F'_{2})$ we also have a homotopy
$\omega^{(3)}_{g}$ between $g^{H_{0}}$ and the constant map
$\bast_{H_{\ast}(F, F'_{1}\sqcup F'_{2})}$. Concatenation 
of the homotopies $\omega^{(1)}_{g}$, $\omega^{(2)}_{g}$, $\omega^{(3)}_{g}$
gives the homotopy $\omega_{g}$.

Notice that 
$$g^{K}=\lambda^{A}_{E}g^{A} = \lambda_{E}g^{Q}$$
It follows that the homotopy $\omega_{g}$ determines a lift of $g^{Q}$:
$$g^{\wh}\colon |\simp|\to \wh(E)_{H_{\ast}(F, F'_{1}\sqcup F'_{2})}$$ 
Let $\bar{g}^{\wh}\colon |\simp|\to \wh(E)$ be the map obtained by reducing
$g^{\wh}$.  One can check that  $\bar{g}^{\wh}$ represents the homotopy type of 
$-j_{0\ast}\tau^{s}(p_{0})$.

In order to construct the homotopy (\ref{ADD MOD EQ})
 we will follow the steps described in  Remark \ref{REDUCTION REM}. 
In the notation of (\ref{REDUCTION REM}) the map $\taus(p)$ is  the reduction of $\ttaus(p)$
and this last map is determined by the pair $(p^{!}, \omega_{p})$. In the same way the map $ j'_{1\ast}
\taus(q_{1})+j'_{2\ast}\taus(q_{2}) + \bar g^{\wh}$ comes from the pair
$$(j'_{1\ast}q_{1}^{!}+j'_{1\ast}q_{1}^{!}+g^{Q}, \omega_{q_{1}}+\omega_{q_{2}}+\omega_{g})$$
Theorem \ref{ADD TRANSFER THM} provides the homotopy $\gamma^{Q}$ between 
$p^{!}$ and $j'_{1\ast}q_{1}^{!}+j'_{1\ast}q_{1}^{!}+g^{Q}$. It follows that to get the homotopy 
(\ref{ADD MOD EQ}) it will suffice to show that 
\begin{itemize}
\item[--] there exists a path $\gamma$ in $K(\R)$ joining the points 
$H_{\ast}(F)$ and $H_{\ast}(F'_{1}) + H_{\ast}(F'_{2}) + H_{\ast}(F, F'_{1}\sqcup F'_{2})$, and
\item[--] there exists a homotopy of homotopies filling the following diagram 
\begin{equation}
\label{ADD HOM HOM EQ}
\xymatrix{
c_{p}\ar[r]^(.3){\lambda_{E} \gamma^{Q}} \ar[d]_{\omega_{p}} & 
c_{q_{1}} + c_{q_{2}}+ g^{K}\ar[d]^{\omega_{q_{1}}+\omega_{q_{2}}+\omega_{g}} \\
\bast_{H_{\ast}(F)}\ar[r]_(.26){\gamma} & 
\bast_{H_{\ast}(F'_{1}) + H_{\ast}(F'_{2}) + H_{\ast}(F, F'_{1}\sqcup F'_{2})} \\
}
\end{equation}
In  this diagram every vertex represents a map $|\simp|\to K(\R)$ 
and every edge represents of homotopy of such maps. 
\end{itemize}

We will first give a combinatorial desctiption of the homotopy $\lambda_{E}\gamma^{Q}$. 
Using Theorem \ref{ADD TRANSFER THM} we get 
$$\lambda_{E}\gamma^{Q}=\lambda^{A}_{E}\ta_{E}\gamma^{Q}=\lambda^{A}_{E}\gamma^{A}$$
Recall that the homotopy $\gamma^{A}$ was obtained applying Waldhausen's pre-additivity theorem
to the functor $\Gamma^{A}$ (\ref{GAMMA FUNCT EQ}). 
Recall also  (\ref{LIN NN}) that the map $\lambda^{A}_{E}$ 
was defined using the linearization functor 
$\Lambda^{A}_{E}\colon \Rfd(E)\to \Chfd(\R)$ (\ref{LAMBDA FUNCTOR EQ}). Let 
$$\Lambda^{A}_{E,2}\colon \TT[2]\Rfd(E)\to \SS[2]\Chfd(\R)$$
be the functor induced by $\Lambda^{A}_{E}$ (see \ref{ALMOST EX NN}). The composition 
$$\Lambda^{A}_{E,2}\circ \Gamma^{A}\colon \simp\to \SS[2]\Chfd(\R)$$ 
is the functor which assigns to a singular simplex 
$\sigma$ the cofibration of chain complexes 
$$\Lambda^{A}_{E,2}\circ \Gamma^{A}(\sigma)=
(C_{\ast}(\sigma^{\ast}E'_{1}\sqcup \sigma^{\ast}E'_{2})\ratail C_{\ast}(\sigma^{\ast}E))$$
Notice that the cofiber of $\Lambda^{A}_{E,2}\circ \Gamma^{A}(\sigma)$ is the chain complex
$G^{K}(\sigma)$ (\ref{GK FUNCT EQ}). 
The homotopy $\lambda^{A}_{E}\gamma^{A}$ is obtained applying Waldhausen's 
pre-additivity theorem (\ref{WAL ADD THM}) to the functor~$\Lambda^{A}_{E,2}\circ \Gamma^{A}$.

Next, recall (\ref{ALG CONTR NN}) that the algebraic contraction $\omega_{p}$ was obtained 
as a concatenation of homotopies $\omega^{(i)}_{p}$, $i=1, 2, 3$.
Consider the following diagram:
\begin{equation}
\label{ADD HOM HOM 3 EQ}
\xymatrix{
c_{p}\ar[r]^(.3){\lambda_{E} \gamma^{Q}} \ar[d]_{\omega^{(1)}_{p}} & 
c_{q_{1}} + c_{q_{2}}+ g^{K}
\ar[d]^{\omega^{(1)}_{q_{1}}+\omega^{(1)}_{q_{2}}+\omega^{(1)}_{g}} \\
h_{p}\ar[r]^(.3){ \gamma^{H}} \ar[d]_{\omega^{(2)}_{p}} & 
h_{q_{1}} + h_{q_{2}}+ g^{H}
\ar[d]^{\omega^{(2)}_{q_{1}}+\omega^{(2)}_{q_{2}}+\omega^{(2)}_{g}} \\
h^{0}_{p}\ar[r]^(.3){\gamma^{H_{0}}} \ar[d]_{\omega^{(3)}_{p}} & 
h^{0}_{q_{1}} + h^{0}_{q_{2}}+ g^{H_{0}}
\ar[d]^{\omega^{(3)}_{q_{1}}+\omega^{(3)}_{q_{2}}+\omega^{(3)}_{g}} \\
\bast_{H_{\ast}(F)}\ar[r]_(.26){\gamma} & 
\bast_{H_{\ast}(F'_{1}) + H_{\ast}(F'_{2}) + H_{\ast}(F, F'_{1}\sqcup F'_{2})} \\
}
\end{equation}
In this diagram every vertex represents a map $|\simp|\to K(\R)$ and every edge 
stands for a homotopy of such maps. The homotopy $\gamma^{H}$ is obtained using 
the homological additivity $\mho^{H}$ as described in (\ref{HOMOLOG ADD NN}).
It is constructed using the homology long exact sequences associated to the short
exact sequences of chain complexes 
\begin{equation}
\label{ADD SES EQ}
C_{\ast}(\sigma^{\ast}E'_{1}\sqcup \sigma^{\ast}E'_{2})\to C_{\ast}(\sigma^{\ast}E)
\to C_{\ast}(\sigma^{\ast}E, \sigma^{\ast}E_{1}'\sqcup \sigma^{\ast}E_{2}')
\end{equation}
The homotopy $\gamma^{H_{0}}$ is obtained in the same manner, while the path 
$\gamma$ is just the restriction of  $\gamma^{H_{0}}$ to the basepoint $b_{0}\in B$.

In order to obtain a homotopy of homotopies filling the diagram (\ref{ADD HOM HOM EQ})
it is enough to show that each of the three squares in the diagram (\ref{ADD HOM HOM 3 EQ})
can be filled with a homotopy of homotopies. In the case of the top square this holds by 
Lemma \ref{ADD THETA LEMMA} (see also \ref{HOMOLOG ADD NN}). 
Existence of a homotopy of homotopies filling the middle square is trivial. 
Finally, a homotopy of homotopies fitting in the bottom square
exists since the long exact sequence of homology groups induced by the  short exact 
sequence 
$$C_{\ast}(F_{1}'\sqcup F_{2}')\to C_{\ast}(F)\to C_{\ast}(F, F_{1}'\sqcup F_{2}')$$
is a sequence of $\pi_{1}B$-modules. 

\end{proof}


\section{The secondary transfer}
\label{SECONDARY TRANSFER SEC }

Our next objective is to develop a formula which, given two suitably good
smooth bundles $p\colon E\to B$ and  $q\colon D\to E$, 
relates the smooth torsion of the bundle $pq$ to the torsions $\taus(p)$ and $\taus(q)$.
We will do it in this section. In Section \ref{TRANSFER AXIOM SEC} we will then verify
that on the level of the cohomological torsion our formula yields the transfer axiom 
(\ref{TRANSFER AX DEF}). 

While Igusa's transfer axiom assumes that the bundle $q$ is a linear oriented sphere bundle
in our context it will be convenient to work in a more general setting: 

\begin{definition}
\label{LH BUNDLE DEF}
Let $q\colon D\to E$ be a smooth bundle with fiber $F$. For $e\in E$ let $F_{e}=p^{-1}(e)$
and let $i_{e}\colon F_{e}\to D$ be the inclusion map. We say that $q$ is a Leray-Hirsch 
bundle if there exists a homomorphism $\theta\colon H^{\ast}(F)\to H^{\ast}(D)$ such that for 
every $e\in E$ the composition 
$$H^{\ast}(F)\overset{\theta}{\lra} H^{\ast}(D)\overset{i_{e}^{\ast}}{\lra} H^{\ast}(F_{e})$$
is an isomorphism.
\end{definition}

\noindent In other words Leray-Hirsch bundles are smooth bundles which satisfy 
the Leray-Hirsch Theorem. For our purposes it will be convenient to state this theorem as follows.

\begin{theorem}[cf. {\cite[Sec. 5.7, Theorem 9]{Spanier}}]
\label{LH THM} Let $q\colon D\to E$ be a Leray-Hirsch bundle with fiber $F$.
There exists an quasi-isomorphism of singular chain complexes 
$$\alpha(q)\colon C_{\ast}(D)\to C_{\ast}(E)\otimes H_{\ast}(F)$$
Moreover, this quasi-somorphism is natural in the following sense: for any map  
$f\colon E'\to E$ the pullback diagram 
$$
\xymatrix{
f^{\ast}D \ar[d]_{q'}\ar[r]^{\bar f} & D \ar[d]^{q} \\
E' \ar[r]_{f} & E \\
}
$$
induces a commutative square
$$
\xymatrix{
C_{\ast}(f^{\ast}D) \ar[d]_{\alpha(q')}\ar[r]^{\bar f_{\ast}} & C_{\ast}(D) \ar[d]^{\alpha(q)} \\
C_{\ast}(E')\otimes H_{\ast}(F) \ar[r]_{f_{\ast}\otimes\id}  & C_{\ast}(E)\otimes H_{\ast}(F) \\
}
$$
\end{theorem}

\begin{remark}
\label{LH REM}
Let $\xi$ be an oriented odd dimensional vector bundle over $E$ and let 
$q\colon D\to E$ be the (even dimensional) sphere bundle associated to $\xi$. 
Then $q$ is a Leray-Hirsch bundle. 
\end{remark}

Let $p\colon E\to B$ be a unipotent bundle and  let  $q\colon D\to E$ 
be a Leray-Hirsch bundle.  As we have mentioned above we will want to  describe the relationship 
between the smooth torsions of $p$, $q$, and $pq$. First, we need to verify that 
$\taus(pq)$ is defined in this case, i.e. that $pq$ is a unipotent bundle. 
This is an immediate consequence of the following 

\begin{lemma}
\label{LH EQUIV LEMMA}
Let $p\colon E\to B$ and $q\colon D\to E$ be smooth bundles with fibers $F_{p}$ and $F_{q}$ respectively, 
and let $F_{pq}$ be the fiber of the bundle $pq\colon D\to B$. Define a $\pi_{1}B$-module structure on 
$H_{\ast}(F_{p})\otimes H_{\ast}(F_{q})$ by
$$\gamma\cdot (x\otimes y):=(\gamma\cdot x)\otimes y$$
for $\gamma\in \pi_{1}B$, $x\otimes y\in H_{\ast}(F_{p})\otimes H_{\ast}(F_{q})$.
If $q$ is a Leray-Hirsch bundle then the Leray-Hirsch isomorphism
$$\alpha(q|_{F_{pq}})\colon H_{\ast}(F_{pq})\to H_{\ast}(F_{p})\otimes H_{\ast}(F_{q})$$
is a $\pi_{1}B$-equivariant map.
\end{lemma}

\noindent Since the identity map $\id_{E}\colon E\to E$ defines a unipotent bundle, 
Lemma \ref{LH EQUIV LEMMA} implies in particular that any Leray-Hirsch bundle is unipotent. 

Next, let $F$ be the fiber of a smooth bundle $q$. Consider the functor 
$$\otimes H_{\ast}(F)\colon \Chfd(\R) \to \Chfd(\R)$$
which assigns to a chain complex $C$ the complex $C\otimes H_{\ast}(F)$. 
This is an exact functor of Waldhausen 
categories, so it induces a map 
$$\otimes H_{\ast}(F)\colon K(\R)\to K(\R)$$

\noindent We have the following

\begin{proposition}
\label{SEC TRANSFER PROP}
Let $q\colon D\to E$ be a Leray-Hirsch bundle with fiber $F$. 
The diagram
$$
\xymatrix{
\tQ(E_{+}) \ar[r]^{\tQ(q^{!})}\ar[d]_{\lambda_{E}} & \tQ(D_{+})\ar[d]^{\lambda_{D}} \\
K(\R) \ar[r]_{\otimes H_{\ast}(F)} & K(\R)
}
$$
commutes up to a preferred homotopy
$$\mu_{q}\colon \tQ(E_{+})\times I \to K(\R)$$
\end{proposition}

\begin{proof} Consider the diagram 
\begin{equation}
\label{A TR DIAG}
\xymatrix{
\tQ(E_{+}) \ar[r]^{\tQ(q^{!})}\ar[d]_{\ta_{E}} & \tQ(D_{+})\ar[d]^{\ta_{D}} \\
A(E) \ar[r]^{A(q^{!})}\ar[d]_{\lambda^{A}_{E}} & A(D)\ar[d]^{\lambda^{A}_{D}} \\
K(\R) \ar[r]_{\otimes H_{\ast}(F)} & K(\R)
}
\end{equation}
where $\ta_{E}, \ta_{D}$ are the assembly maps and $\lambda^{A}_{E}, \lambda^{A}_{D}$ are 
the $A$-theory linearization maps (see \ref{LIN NN}). Recall that we have 
$\lambda_{E}=\lambda^{A}_{E}\circ \ta_{E}$ and  $\lambda_{D}=\lambda^{A}_{D}\circ \ta_{D}$.
The map $A(q^{!})$ is the $A$-theory transfer of $q$. It is induced by the exact functor
$\Rfd(E)\to \Rfd(D)$ which associates to a retractive space $X\in \Rfd(E)$ the space 
$$q^{\ast}X=\lim(X\to E \overset{q}{\la} D)$$

Directly from the constructions of \cite[Sec. 3]{BDW} it follows that the upper square in the diagram 
(\ref{A TR DIAG}) commutes. Consequently, it will suffice to construct a homotopy 
$$\mu^{A}_{q}\colon A(E)\times I\to K(\R)$$
which makes the lower square commute. We will then define
$$\mu_{q}:=\mu^{A}_{q}\circ(\ta_{E}\times \id)$$

Notice that the map $\lambda^{A}_{D}\circ\nolinebreak A(q^{!})$ is induced by the 
functor 
$$F\colon \Rfd(E)\to \Chfd(\R)$$
which assigns to a retractive space $X\in \Rfd(E)$ the singular chain complex 
$C_{\ast}(q^{\ast}X)$. The map $\otimes H_{\ast}(F)\circ \lambda^{A}_{E}$, on the other hand, comes 
from the functor 
$$G\colon \Rfd(E)\to \Chfd(\R)$$
such that $G(X)= C(X)\otimes H_{\ast}(F)$. 

Notice also that for any $X\in \Rfd(E)$ the map $q^{\ast}X\to X$ is a Leray-Hirsch bundle induced 
from $q\colon D\to E$. By Theorem \ref{LH THM} we have a quasi-isomorphism 
$$\alpha(q^{\ast}X\to X)\colon C_{\ast}(q^{\ast}X)\to C_{\ast}(X)\otimes H_{\ast}(F)$$
 By naturality of the quasi-isomorphisms $\alpha(-)$
this defines a natural transformation
$$\alpha\colon F\Rightarrow G$$
This natural transformation induces the desired homotopy $\mu^{A}_{q}$.

\end{proof}

\begin{definition}
\label{SEC TRANSFER DEF} 
Let $q\colon D\to E$ be a Leray-Hirsch bundle with fiber $F$.
The homotopy   $\mu_{q}$ defines a map $\wh(q^{!})\colon \wh(E)\to \wh(D)$ 
such that we get a homotopy commutative diagram
$$
\xymatrix{
\wh(E)\ar[d] \ar[r]^{\wh(q^{!})} & \wh(D) \ar[d] \\
\tQ(E_{+}) \ar[r]^{\tQ(q^{!})}\ar[d]_{\lambda_{E}} & \tQ(D_{+})\ar[d]^{\lambda_{D}} \\
K(\R) \ar[r]_{\otimes H_{\ast}(F)} & K(\R)
}
$$
We will call $\wh(q^{!})$ the secondary transfer of $q$.
\end{definition}

\begin{remark}
\label{CHI MULT REM}
The following observation will be useful later on.
Let $F$ be the fiber of  a Leray-Hirsch bundle $q\colon D\to E$ and let $\chi(F)\in\Z$ 
be the Euler characteristic of $F$. Notice that in the infinite loop space structure on 
$K(\R)$ the map $\otimes H_{\ast}(F)\colon K(\R)\to K(\R)$ represents multiplication 
by $\chi(F)$. As a consequence we have a homotopy commutative diagram
$$
\xymatrix{
\Omega K(\R) \ar[r]^{\cdot \chi(F)} \ar[d] & \Omega K(\R) \ar[d] \\
\wh(E) \ar[r]_{\wh(q^{!})} & \wh(D) \\
}
$$
\end{remark}

For our purposes the key property of the secondary transfer is given by the following

\begin{theorem}
\label{SEC TRANSFER THM}
Let $q\colon D\to E$ be a Leray-Hirsch bundle  and let $p\colon E\to B$
be a unipotent bundle.  The following diagram commutes up to homotopy
$$
\xymatrix{
& \wh(E)\ar[d]^{\wh(q^{!})} \\
|\simp|\ar[ur]^{\taus(p)}\ar[r]_{\taus(pq)} & \wh(D) \\
}
$$
\end{theorem}

\begin{proof}
Let $F_{p}$, $F_{q}$ and $F_{pq}$ be the fibers of the bundles $p$, $q$, and $pq$ respectively.
Consider the following homotopies: 
\begin{itemize}
\item $\mu_{q}\circ (p^{!} \times \id_{I})\colon |\simp|\times I \to K(\R)$ is a homotopy between the maps 
$\lambda_{D}\circ \tQ(q^{!})\circ p^{!}$ and $\otimes H_{\ast}(F_{q})\circ \lambda_{E}\circ p^{!}$. 
By the construction of $p^{!}$ and $\tQ(q^{!})$ in \cite{BDW} we have $\tQ(q!) \circ p^{!}=(pq)^{!}$ thus
$$\lambda_{D}\circ \tQ(q^{!})\circ p^{!}= c_{pq}$$
Also,  since $ \lambda_{E}\circ p^{!}=c_{p}$ we obtain
$$\left\{ \begin{array}{lcl}
\mu_{q}\circ (p^{!}\times \id_{I})|_{|\simp|\times \{0\}}& = & c_{pq}\\
\mu_{q}\circ (p^{!}\times \id_{I})|_{|\simp|\times \{1\}} & =& c_{p}\otimes H_{\ast}(F_{q})\\
\end{array}\right. 
$$  

\vskip 1mm
\item The homotopy $\omega_{p}\otimes H_{\ast}(F_{q})\colon |\simp|\times I \to K(\R)$ is obtained
using the algebraic contraction of the bundle $p$ (\ref{ALG CONTR NN}). We have
$$\left\{ \begin{array}{lcl}
\omega_{p}\otimes H_{\ast}(F_{q}))|_{|\simp|\times \{0\}}& = & c_{p}\otimes H_{\ast}(F_{p})\\
\omega_{p}\otimes H_{\ast}(F_{q})|_{|\simp|\times \{1\}} & =&\bast_{H(F_{p})\otimes H(F_{q})}\\
\end{array}\right. 
$$  
where  $\bast_{H(F_{p})\otimes H_{\ast}(F_{q})}$ is the constant map sending $|\simp|$ to the point 
$H(F_{p})\otimes H(F_{q})\in K(\R)$.
\end{itemize}
Let $\psi_{pq}$ denote the homotopy obtained by concatenating these two homotopies. 
The pair $((pq)^{!}, \psi_{pq})$ determines a map 
$$\kappa(pq) \colon |\simp|\to \wh(D)_{H_{\ast}(F_{p})\otimes H_{\ast}(F_{q})}$$ 
where as in (\ref{WH DEF}) we set
$$\wh(D)_{H_{\ast}(F_{p})\otimes H_{\ast}(F_{q})}:=
\hofib(\lambda_{D}\colon \tQ(D_{+})\to K(\R))_{H_{\ast}(F_{p})\otimes H_{\ast}(F_{q})}$$
The map $\wh(q^{!})\circ \taus(p)$ is homotopic to the map obtained by reducing $\kappa(pq)$
i.e. by shifting it to a map $|\simp|\to \wh(D)$
(cf. \ref{REDUCTION REM}). The map $\taus(pq)$ is, in turn, the reduction of the unreduced
torsion 
$$\ttaus(pq)\colon |\simp|\to \wh(D)_{H_{\ast}(F_{pq})}$$ 
which is determined by the pair $((pq)^{!}, \omega_{pq})$.
In order to show that $\taus(pq)\simeq \wh(q^{!})\circ \taus(p)$ we can 
now proceed following the steps described in Remark \ref{REDUCTION REM}. We need:
\begin{itemize}
\item[-] a homotopy $(pq)^{!}\simeq (pq)^{!}$ - our choice is the trivial one;
\item[-] a path in $K(\R)$ joining the points $H_{\ast}(F_{pq})$ and $H_{\ast}(F_{p})\otimes H_{\ast}(F_{q})$. 
Such path is defined by the Leray-Hirsch isomorphism
$$\alpha(q|_{F_{pq}})\colon H_{\ast}(F_{pq})\to H_{\ast}(F_{p})\otimes H_{\ast}(F_{q})$$
\item[-] a homotopy of homotopies filling the diagram 
\begin{equation}
\label{COMP HOMOT DIAG}
\xymatrix{
c_{pq} \ar@{=}[rr] \ar[d]_{\omega_{pq}} & &c_{pq} \ar[d]^{\psi_{pq}} \\
\bast_{H_{\ast}(F_{pq})} \ar[rr]_{\alpha(q|_{F_{pq}})} &  &\bast_{H_{\ast}(F_{p})\otimes H_{\ast}(F_{q})} \\
}
\end{equation}
\end{itemize}
To obtain such homotopy of homotopies  recall (\ref{ALG CONTR NN}) that for a unipotent bundle $p$ the 
algebraic contraction $\omega_{pq}$ is the concatenation of  homotopies $\omega^{(i)}_{p}$, $i=1, 2, 3$.
Consider the following diagram
\begin{equation}
\label{SEC HOMOT DIAG}
\xymatrix{
c_{pq} \ar@{=}[rr] \ar[dd]_{\omega^{(1)}_{pq}} & & 
c_{pq} \ar[d]^{\mu_{q}\circ (p^{!}\times\id_{I})} \\
 & & 
 c_{p}\otimes H_{\ast}(F_{q})\ar[d]^{\omega^{(1)}_{p}\otimes H_{\ast}(F_{q})} \\
h_{pq} \ar[rr]^{\alpha(q)}\ar[d]_{\omega^{(2)}_{pq}} & &
h_{p}\otimes H_{\ast}(F_{q}) \ar[d]^{\omega^{(2)}_{p}\otimes H_{\ast}(F_{q})}\\
h^{0}_{pq} \ar[rr]_{\alpha^{0}(q)} \ar[d]_{\omega^{(3)}_{pq}} & &
h^{0}_{p}\otimes H(F_{q}) \ar[d]^{\omega^{(3)}_{p}\otimes H_{\ast}(F_{q})} \\ 
H(F_{pq}) \ar[rr]_{\alpha(q|_{F_{pq}})} & &
H(F_{p})\otimes H(F_{q}) \\
}
\end{equation}
In this diagram every vertex represents a map $|\simp|\to K(\R)$, and every edge represents a homotopy 
of such maps. Concatenation of the vertical homotopies on the left gives $\omega_{pq}$, 
while concatenating the vertical homotopies on the right we obtain $\psi_{pq}$.
Recall that the map $h_{pq}$ is induced by the functor $H_{pq}\colon \simp\to \SS[1]\Chfd(\R)$
given by $H_{pq}(\sigma)=H_{\ast}(\sigma^{\ast}D)$. Since $q$ is a Leray-Hirsch bundle for every 
$\sigma\in \simp$ we have the Leray-Hirsch isomorphism 
$H_{\ast}(\sigma^{\ast}D)\overset{\cong}{\lra}H_{\ast}(\sigma^{\ast}E)\otimes H_{\ast}(F_{q})$. 
These isomorphisms define a natural transformation of functors 
$$H_{pq}\Rightarrow H_{p}\otimes H_{\ast}(F_{q})$$
This natural transformation induces the homotopy $\alpha(q)$ between the maps 
$h_{pq}$ and $h_{p}\otimes H(F_{q})$. The homotopy $\alpha^{0}(q)$ between 
$h^{0}_{pq}$ and $h^{0}_{p}\otimes H_{\ast}(F_{q})$ is obtained in the same way.

In order to obtain a homotopy of homotopies filling the diagram (\ref{COMP HOMOT DIAG})
it suffices to show that each of the squares of the diagram (\ref{SEC HOMOT DIAG}) can be filled 
by a homotopy of homotopies. 
Such homotopy of homotopies filling the top square is described in
(\ref{TOP SQ PROP}). By the naturality of the Leray-Hirsch isomorphisms the concatenations
of $\alpha(q)$ with $\omega^{(2)}_{p}\otimes H_{\ast}(F_{q})$ and $\omega_{pq}$ with $\alpha^{0}(q)$
actually conincide, so the middle square is trivially filled by a homotopy of homotopies.
Finally, existence of homotopy of homotopies filling 
the bottom square follows directly from the naturality and $\pi_{1}B$-equivariance of the Leray-Hirsch 
isomorphism (\ref{LH EQUIV LEMMA}).
\end{proof}


\section{The transfer axiom}
\label{TRANSFER AXIOM SEC}

The results of the last section can be used to show that the formula of Igusa's transfer axiom 
(\ref{TRANSFER AX DEF}) is satisfied by the smooth cohomological torsion whenever $p$
is a unipotent bundle and $q$ is any Leray-Hirsch bundle. 

\begin{theorem}
\label{LH TRANSFER THM}
Let $q\colon D\to E$ be a Leray-Hirsch bundle with fiber $F$ and let $p\colon E\to B$
be a unipotent bundle. We have
$$\ts(pq)=\chi(F) \ts(p) + tr^{E}_{B}(\ts(q))$$
\end{theorem}

\begin{proof}
Let  $\wh(\ast)_{\Q}$ be the rationalization of  the space $\wh(\ast)$, and let 
$${\bar\iota}_{D} \colon \wh(D)\to \wh(S^{0})_{\Q}$$ 
be the map as in (\ref{COH STOR NN}).
Recall (\ref{COH STOR DEF}) that the cohomology class $\ts(pq)$ is represented by the map 
$${\bar\iota}_{D}\taus(pq)\colon |\simp| \to \wh(S^{0})_{\Q}$$
Since $\wh(D)$ is an infinite loop space the map $\taus(q)\colon |\simp[E]|\to \wh(D)$ admits the extension
$\tQ(E_{+})\to \wh(D)$ which, by abuse of notation, we will also denote by $\taus(q)$. 
The cohomology class $tr^{E}_{B}(\ts(q))$ is represented by the map
$${\bar\iota}_{D}\taus(q)p^{!}\colon |\simp| \to \wh(S^{0})_{\Q}$$
By Theorem \ref{SEC TRANSFER THM} we have 
$${\bar\iota}_{D}\taus(pq)\simeq {\bar\iota}_{D}\wh(q^{!})\taus(p) \text{ \hskip 5mm and \hskip 5 mm} 
{\bar\iota}_{D}\taus(q)p^{!}\simeq {\bar\iota}_{D}\wh(q^{!})\taus(\id_{E})p^{!}$$
where $\taus(\id_{E})$ is the smooth torsion of the identity bundle $\id_{E}\colon E\to E$.
It follows that the cohomology class $$\ts(pq) - tr^{E}_{B}(\ts(q))$$ is represented by 
the map ${\bar\iota}_{D}\wh(q^{!})\varrho$ where
\begin{equation}
\label{PSI EQ}
\varrho := \taus(p)-\taus(\id_{E})p^{!}
\end{equation}
It suffices to show that the map ${\bar\iota}_{D}\wh(q^{!})\varrho$  represents also the cohomology class
$\chi(F_{q})\ts(p)$.

Consider the diagram 
$$
\xymatrix{
& \Omega K(\R) \ar[r]^{\cdot \chi(F)} \ar[d]_{j_{E}} & \Omega K(\R)\ar[d]^{j_{D}} & \\
|\simp| \ar@{-->}[ur]^{\tilde \varrho}\ar[r]_{\varrho} \ar[rd]_{{\bar p}^{!}-{\bar p}^{!}} 
& \wh(E) \ar[r]_{\wh(q^{!})} \ar[d]& \wh(D)\ar[r]_{{\bar\iota}_{D}}\ar[d] & \wh(S^{0})_{\Q}\\
& \tQ(E_{+}) \ar[r]_{\tQ(q^{!})} &  \tQ(D_{+})&\\
}
$$
The two pairs of vertical maps are fibration sequences. The lower square in the diagram commutes 
up to homotopy by the definition of  $\wh(q^{!})$ (\ref{SEC TRANSFER DEF}), 
and the upper square is homotopy commutative by Remark \ref{CHI MULT REM}.
Notice that both $\taus(p)$ and $\taus(\id_{E})p^{!}$ are lifts of the reduced Becker-Gottlieb transfer map 
${\bar p}^{!}\colon |\simp|\to Q(E_{+})$.
As a consequence $\varrho$ is a lift of the contractible map ${\bar p}^{!}-{\bar p}^{!}$, 
and so it admits a lift $\tilde \varrho\colon |\simp|\to \Omega K(\R)$. 
Homotopy commutativity of the upper square gives
$${\bar\iota}_{D}\wh(q^{!})\varrho\simeq ({\bar\iota}_{D}j_{D}\tilde \varrho)\cdot \chi(F)$$
Next, homotopy commutativity of the diagram (\ref{COH DIAG EQ}) implies that the following 
diagram commutes up to homotopy:
$$
\xymatrix{
\Omega K(\R) \ar@{=}[r] \ar[d]_{j_{E}}& \Omega K(\R) \ar[d]^{j_{\ast}} & \Omega K(\R) \ar@{=}[l] \ar[d]^{j_{D}}\\
\wh(E) \ar[r]_{\bar\iota_{E}} & \wh(\ast) & \wh(D) \ar[l]^{\bar\iota_{D}} \\
}
$$
In particular we have ${\bar\iota}_{E}j_{E} \simeq {\bar\iota}_{D}j_{D}$. This gives
\begin{equation}
\label{PSI2 EQ}
{\bar\iota}_{D} \wh(q^{!})\varrho \simeq ({\bar\iota}_{E}j_{E}\tilde \varrho)\cdot \chi(F_{q})\simeq 
({\bar\iota}_{E}\varrho)\cdot \chi(F)
\end{equation}
By the definition (\ref{PSI EQ}) of the map $\varrho$ we have 
$${\bar\iota}_{E}\varrho \simeq {\bar\iota}_{E}\taus(p)-{\bar\iota}_{E}\taus(\id_{E})p^{!}$$
The maps ${\bar\iota}_{E}\taus(p)$ and ${\bar\iota}_{E}\taus(\id_{E})p^{!}$ represent, respectively, 
the cohomology classes $\ts(p)$ and $tr^{E}_{B}(\ts(\id_{E}))$. Since $\id_{E}$ is a trivial bundle, 
we have $\ts(\id_{E})=0$. It follows that ${\bar\iota}_{E}\varrho$ represents the class $\ts(p)$.  
Combining this with (\ref{PSI2 EQ}) we obtain that the cohomology class $\chi(F)\cdot \ts(p)$ 
is represented by the map  ${\bar\iota}_{D}\wh(q^{!})\varrho$. 
\end{proof}

Theorem \ref{LH TRANSFER THM} does not by itself show that Igusa's transfer axiom 
is satisfied by the smooth cohomological torsion $t^{s}$. 
Indeed, if $q$ is an odd dimensional oriented linear  sphere bundle then $q$
need not be a Leray-Hirsch bundle. The statement of Theorem  \ref{LH TRANSFER THM}, however, 
can be extended to the case where $q$ any  linear oriented sphere bundle.

\begin{corollary}
\label{TRANSFER COR}
Let $p\colon E\to B$ be a unipotent bundle and let $q\colon D\to E$ be an oriented linear sphere bundle with 
fiber $S^{n}$. We have
$$\ts(pq)=\chi(S^{n})\ts(p)+tr^{E}_{B}(\ts(q))$$
\end{corollary}

\begin{proof}
If $n$ is even the statement follows from Remark \ref{LH REM} and Theorem \ref{LH TRANSFER THM}. 
Assume then that $n$ is odd, $n=2k-1$, and let  $\xi^{2k}$  be  the oriented vector bundle such that 
 $q\colon D\to E$ is the sphere bundle associated to $\xi$. 
 Let $\varepsilon^{1}$ be the 
 1-dimensional trivial vector bundle over $E$ and let 
 $$q' \colon S(\xi\oplus \varepsilon^{1})\to E$$
 be the $2k$-dimensional sphere bundle associated to $\xi\oplus \epsilon^{1}$. The bundle
 $q'$ admits a splitting 
 $$q' \cong D(q)\cup_{q} D(q)$$
where $D(q)$ is the disc bundle associated to $\xi$. Similarly, if $p\colon E\to B$ is a unipotent
bundle then we have a splitting 
$$pq'\cong(pD(q))\cup_{pq}(pD(q)) $$

Additivity of  the smooth cohomological torsion  (Corollary \ref{ADD THM}) 
implies that 
$$\ts(pq')=2\ts(pD(q))-\ts(pq) \ \ \ \mbox{and}\ \ \ \  \ts(q')= 2\ts(D(q))-\ts(q)$$
Also, since $\ts$ is an exotic invariant (\ref{EXOTIC DEF})
we obtain $\ts(pD(q))=\ts(p)$ and $\ts(D(q))=0$. This gives

\begin{equation}
\label{QQ' EQ}
\ts(pq')=2\ts(p)-\ts(pq) \ \ \ \mbox{and}\ \ \ \  \ts(q')= -\ts(q)
\end{equation}
By  (\ref{LH REM}) the bundle $q'$ is a Leray-Hirsch bundle, so applying
Theorem \ref{LH TRANSFER THM} we get
$$\ts(pq')=\chi(S^{2k})\cdot \ts(p)+ tr^{E}_{B}(\ts(q'))=2\ts(p)+ tr^{E}_{B}(\ts(q'))$$
Combining this with the equations (\ref{QQ' EQ}) we obtain
$$\ts(pq)=-tr^{E}_{B}(\ts(q'))=tr^{E}_{B}(\ts(q))$$
Since $\chi(S^{2k-1})=0$ this gives
$$\ts(pq)=\chi(S^{2k-1})\cdot \ts(p)+ tr^{E}_{B}(\ts(q))$$

\end{proof}

\section{Nontriviality of the smooth torsion}
\label{NONTRIVIALITY SEC}

Recall (\ref{MAIN THM}) that by $t^{s}_{4k}$ we denoted the degree $4k$ component of the 
smooth cohomological torsion $t^{s}$. Our results so far show that $t^{s}_{4k}$ is an exotic
higher torsion invariant of unipotent bundles. Our final goal is to demonstrate that 
$t^{s}_{4k}$ is a non-trivial invariant. 

\begin{theorem}
\label{NONTRIVIALITY THM}
For each $k>0$ there exists a smooth bundle $p\colon E\to S^{4k}$ such that $t^{s}_{4k}(p)\neq 0$.
\end{theorem}

Let $BSO$ be the classifying space of oriented virtual vector bundles of dimension 0, and let 
$BSG$ the classifying space of 0 dimensional virtual oriented spherical fibrations. We have 
the $J$-homomorphism map 
$$J\colon BSO\to BSG$$
The proof of Theorem \ref{NONTRIVIALITY THM} will rely on the following 

\begin{lemma}
\label{WAL J LEMMA}
There exists a homotopy commutative diagram
$$
\xymatrix{
BSO \ar[r]^{\eta}\ar[d]_{J} & \tQ(S^{0})\ar[d]^{\ta_{\ast}} \\
BSG \ar[r]_{\eta'} & A(\ast) \\
}
$$  
Moreover, if $p\colon E\to B$ is an oriented linear sphere bundle classified by a map 
$f\colon B \to BSO$ then the diagram 
$$
\xymatrix{
B \ar[r]^{p^{!}}\ar[d]_{f} & \tQ(E_{+}) \ar[d] \\
BSO \ar[r]_{\eta} & \tQ(S^{0}) \\
}
$$
commutes up to homotopy.
\end{lemma}

\begin{proof}
The first part of the lemma follows from \cite{WalM1}, Propositions 3.1 and 3.2.
The second part is a consequence of \cite{Bok-Wal}, Lemma 2.6.
\end{proof}

\begin{proof}[Proof of Theorem \ref{NONTRIVIALITY THM}]
Fix $k>0$. By the Bott periodicity we have $\pi_{4k}BSO\cong \Z$. Since the the homotopy 
group $\pi_{4k}BSG$ is finite the kernel of the homomorphism
$$J_{\ast}\colon \pi_{4k}(BSO)\to \pi_{4k}(BSG)$$
contains torsion free elements. Let $f\colon S^{4k}\to BSO$ be a map representing 
such element in $\ker J_{\ast}$ and let $p\colon E\to S^{4k}$ be a linear oriented sphere bundle 
classified by $f$. We will show that $t^{s}_{4k}(p)\neq 0$. 

Consider the homotopy commutative diagram
$$
\xymatrix{
& G/O \ar[r]\ar[d] & \Omega \wh[\rm diff](\ast) \ar[r]\ar[d] & \wh(\ast)\ar[d] \\
S^{4k}\ar[r]^{f} \ar@{-->}[ru]^{\tilde f}& BSO \ar[r]^{\eta}\ar[d]_{J} & 
\tQ(S^{0})  \ar@{=}[r]\ar[d]^{\ta_{\ast}} & \tQ(S^{0})\ar[d]^{\lambda_{\ast}} \\
& BSG \ar[r]^{\eta'} & A(\ast) \ar[r]^{\lambda^{A}_{\ast}} & K(\R) \\
}
$$
where each pair of  vertical maps is a fibration sequence. By the choice of $f$ the composition
$J\circ f$ is homotopy trivial, so  it admits a lift to $\tilde f\colon S^{4k}\to G/O$.
Using Lemma \ref{WAL J LEMMA} one can check that the composition 
\begin{equation}
\label{NONTRIV EQ}
|\simp[S^{4k}]|\overset{\simeq}{\lra}S^{4n}\overset{\tilde f}{\lra} G/O \to \wh(\ast)
\end{equation}
is homotopic to the map 
$$|\simp[S^{4n}]|\overset{\tau^{s}(p)}{\lra} \wh(E)\to \wh(\ast)$$
As a consequence in order to show that $t^{s}_{4k}(p)\neq 0$ we only need to verify 
that the map (\ref{NONTRIV EQ}) is rationally non-trivial. 
By the choice of $f$ the map $\tilde f$ represents a torsion free element in $\pi_{4k}(G/O)$, 
thus if $(G/O)_{\Q}$ is the rationalization of $G/O$ then the map 
$$\tilde f\colon S^{4k} \to (G/O)_{\Q}$$
is not contractible. It is then enough to notice that  the map 
$$(G/O)_{\Q}\to \wh(\ast)_{\Q}$$
yields a monomorphism on the level of homotopy groups. This  follows essentially from 
\cite[Theorem 1]{Bokstedt} and \cite[Proposition 2.2]{Wal1}.
\end{proof}


\renewcommand{\thesubsection}{\setcounter{theorem}{0}{A.\arabic{subsection}}}

\makeatletter
\def\subsection#1{
  \ifhmode\par\fi
  \removelastskip
  \vskip 3ex\goodbreak
  \refstepcounter{subsection}%
  \noindent
  \leavevmode
  \begingroup
  \bfseries
  \thesubsection.\ #1.
  \endgroup
}
\makeatother

\renewcommand{\theequation}{A-\arabic{equation}} 
\renewcommand{\thetheorem}{A.\arabic{subsection}.\arabic{theorem}}
\setcounter{theorem}{0}

\setcounter{secnumdepth}{2}

\section*{Appendix : Homology of chain complexes}
\label{APP SEC}

In this appendix we gathered some facts related to the passage from chain complexes 
to their homology on the level of $K$-theory which we use throughout the paper.

\subsection{The homotopy $\bm\Theta$}
Let $H\colon \SS[1]\Chfd(\R) \to \SS[1]\Chfd(\R)$ be the functor which assigns to each chain complex 
$C$ its homology complex $H_{\ast}(C)$, and let 
$$|H|\colon |\SS[1]\Chfd(\R)|\to |\SS[1]\Chfd(\R)|$$ 
be the map induced by $H$ on the nerve of $\SS[1]\Chfd(\R)$. 
Also, let  
$$k\colon |\SS[1]\Chfd(\R)|\to K(\R)$$ 
be the map as in (\ref{ST MAP NN}).  We have

\begin{lemma}[{\cite[Lemma 6.8]{BDW}}]
\label{HOMOLOGY LEMMA}
The maps $k$ and $k \circ |H|$ are homotopic via a preferred homotopy
$$\Theta \colon |\SS[1]\Chfd(\R)|\times I \to K(\R)$$

\end{lemma}

\begin{proof}
For a chain complex 
$$
\xymatrix{
C = (\dots \ar[r]  & C_{2}\ar[r]^{\partial_{2}} & C_{1}\ar[r]^{\partial_{1}} & C_{0}\ar[r] & 0 )
}
$$ 
let $P_{q}C$ denote the complex such that $(P_{q}C)_{i} =0$ for $i>q+1$, 
$(P_{q}C)_{q+1}= \partial(C_{q+1})$, and $(P_{q}C)_{i}=C_{i}$ for $i\leq q$. 
Let $Q_{q}C$ be the kernel of the map $P_{q}C\to P_{q-1}C$. We obtain cofibration
sequences of chain complexes
$$Q_{q}C\to P_{q}C\to P_{q-1}C$$
functorial in $C$. 
Notice that the complex $Q_{q}C$ is canonically quasi-isomorphic to its ho\-mo\-logy complex 
$H_{\ast}(Q_{q}C)$, and that this last complex has only one non-zero module $H_{q}(C)$ in degree $q$. 
Let $P_{q}\colon \SS[1]\Chfd(\R)\to \SS[1]\Chfd(\R)$ be the functor which assigns to 
to $C\in \SS[1]\Chfd(\R)$ the chain complex $P_{q}C$.
By Waldhausen's pre-additivity theorem (\ref{WAL ADD THM})
the map $k\circ |P_{q}|\colon |\SS[1]\Chfd(\R)|\to K(\R)$
is homotopic to the map induced by the assignment 
$$C\mapsto P_{q-1}C \oplus H_{\ast}(Q_{q}C)$$
Iterating this argument  for each  $q$ we obtain a homotopy 
$$\Theta_{q}\colon |\SS[1]\Chfd(\R)|\times I \to K(\R)$$ 
between the map  $k\circ |P_{q}|$ and $k\circ |H|\circ |P_{q}|$.
Finally, notice that since chain complexes in $\Chfd(\R)$ are homotopy finitely dominated,
for any $C\in \Chfd(\R)$ there is $q\geq 0$ such that $P_{q}C\simeq C$. This means that on
the connected component of $C$ in $|\SS[1]\Chfd(\R)|$ we can set $\Theta := \Theta_{q}$.
\end{proof}

\begin{remark}
\label{OMEGA1 REM}
Let $p\colon E\to B$ be a smooth bundle, and let 
$$c_{p}, h_{p}\colon |\simp|\to K(\R)$$
be the maps defined, respectively, in (\ref{LIN NN}) and (\ref{ALG CONTR NN}). 
Let $C_{p}\colon \simp \to \SS[1]\Chfd(\R)$ be the 
chain complex functor (\ref{LIN NN}) and let
$$|C_{p}|\colon |\simp|\to |\SS[1]\Chfd(\R)|$$
be the map of nerves of categories induced by $C_{p}$.
Notice that  $c_{p}=k \circ |C_{p}|$ and $h_{p}=k \circ |H|\circ |C_{p}|$. As a consequence applying 
Lemma \ref{HOMOLOGY LEMMA} be obtain a homotopy between the maps $c_{p}$ and $h_{p}$. 
This is the homotopy $\omega^{(1)}_{p}$ used in (\ref{ALG CONTR NN}) to construct the algebraic 
contraction for $p$.
\end{remark}

\subsection{\bf Homological additivity} 
\label{HOMOLOG ADD}
The category $\SS[2]\Chfd(\R)$ (\ref{WAL CATS NN})
can be identified with the category of sort exact sequences of chain complexes with 
quasi-isomorphisms of short exact sequences as morphisms. For a short exact sequence 
$$S:=(0\to A \lra B \lra C \to 0)$$
we have then $Ev_{1}(S)=A$, $Ev_{2}(S)=B$, and $Ev_{12}(S)=C$
where 
$$Ev_{i}\colon \SS[2]\Chfd(\R)\to \SS[1]\Chfd(\R)$$ 
are the functors defined  in (\ref{WAL ADD NN}).  Let  
$$\Theta^{i}\colon |\SS[2]\Chfd(\R)|\times I\to K(\R)$$ 
be the homotopy between 
$k\circ |Ev_{i}|$ and $k\circ |H| \circ |Ev_{i}|$ defined by $\Theta$.
Consider the following diagram
\begin{equation}
\label{ADD THETA EQ}
\xymatrix{
k\circ |Ev_{2}| \ar[r]^(.4){\mho}\ar[d]_{\Theta^{2}} & 
k\circ |Ev_{1}|+ k\circ |Ev_{12}|\ar[d]^{\Theta^{1}+\Theta^{12}} \\
k\circ |H|\circ |Ev_{2}|&  k\circ |H|\circ |Ev_{1}|+ k\circ |H| \circ |Ev_{12}| \\
}
\end{equation}
Here every vertex represents a map $|\SS[2]\Chfd(\R)|\to K(\R)$ and each arrow 
is a homotopy of such maps. The homotopy $\mho$ is given by Waldhausen's pre-additivity theorem 
(\ref{WAL ADD THM}). Concatenation of homotopies appearing in the diagram (\ref{ADD THETA EQ})
defines a homotopy
$$k\circ |H|\circ |Ev_{2}| \simeq  k\circ |H|\circ |Ev_{1}|+ k\circ |H| \circ |Ev_{12}|$$
This homotopy can be described more directly as follows. Given a short exact sequence of chain 
complexes 
$$0\to A\overset{f}{\lra} B \overset{g}{\lra} C\to 0 $$
consider the associated long exact sequence of homology groups 
$$\dots\  \overset{\delta_{q+1}}{\lra} H_{q}(A)\overset{f_{q}}{\lra} H_{q}(B) \overset{g_{q}}{\lra} H_{q}(C)
\overset{\delta_{q}}{\lra}  H_{q-1}(A) \overset{f_{q-1}}{\lra} \dots$$
Let $\ker f$  denote the chain complex with trivial differentials given by  
$$(\ker f)_{q}:= \ker(f_{q}\colon H_{q}(A)\to H_{q}(B))$$
and let $\ker g$, $\ker \delta$ be chain complexes defined in the analogous way.
Notice that we have a short exact sequence 
$$0\to \ker g \lra H_{\ast}(B)\lra \ker \delta \to 0$$
and so the homotopy $\mho$ defines a path in $K(\R)$ between points represented by the complexes
$H_{\ast}(B)$ and $\ker g\oplus \ker \delta$. By an analogous argument we obtain a path in 
$K(\R)$ joining the points represented by $H_{\ast}(A)\oplus H_{\ast}(C)$ and by the chain complex
$$(\ker f\oplus \ker g)\oplus (\ker \delta \oplus \ker f[-1])$$
where  $\ker f[-1]$ is the complex obtained by shifting grading in $\ker f$:
$$(\ker f[-1])_{q}:=(\ker f)_{q-1}$$
Using the additivity homotopy $\mho$ again one can obtain a canonical path in $K(\R)$ 
between the points represented by $\ker f\oplus \ker f[-1]$ and the zero chain complex. 
Combining it with the other paths described above we get a path in $K(\R)$:
$$\xymatrix{
& \ker g\oplus \ker \delta \ar@{-}[dd] & \\
H_{\ast}(B) \ar@{-}[ur] & &  H_{\ast}(A)\oplus H_{\ast}(C) \\
& \ker f\oplus \ker g\oplus \ker \delta \oplus \ker f[-1] \ar@{-}[ur] & \\
}
$$
Notice that all steps in the construction of this path are functorial, so in fact we obtain in this 
way a homotopy  
$$\mho^{H}\colon |\SS[2]\Chfd(\R)|\times I\to K(\R)$$ 
between the map $k\circ |H|\circ |Ev_{2}|$ and  $k\circ |H|\circ |Ev_{1}|+ k\circ |H| \circ |Ev_{12}|$. 
The homotopy $\mho^{H}$ extends the diagram (\ref{ADD THETA EQ}):
\begin{equation}
\label{ADD THETA2 EQ}
\xymatrix{
k\circ |Ev_{2}| \ar[rr]^(.4){\mho}\ar[d]_{\Theta^{2}}&  & 
k\circ |Ev_{1}|+ k\circ |Ev_{12}|\ar[d]^{\Theta^{1}+\Theta^{12}} \\
k\circ |H|\circ |Ev_{2}|\ar[rr]_(.4){\mho^{H}}& & 
k\circ |H|\circ |Ev_{1}|+ k\circ |H| \circ |Ev_{12}| \\
}
\end{equation}

\begin{lemma}
\label{ADD THETA LEMMA}
There exists a homotopy of homotopies which fills the dia\-gram (\ref{ADD THETA2 EQ}).
\end{lemma}

\noindent A proof of the lemma can be obtained by a fairly straighforward (although tedious) 
construction of the required homotopy of homotopies.

\begin{nn}
\label{HOMOLOG ADD NN}
Let $\DD$ be a small category and let 
$$F\colon \DD\to \SS[2]\Chfd(\R)$$
be a functor which associates to $d\in \DD$ a short exact sequence 
$$F(d)=(0\to A(d)\to B(d)\to C(d)\to 0)$$
As before we identify here $\SS[2]\Chfd(\R)$ with the category of short exact sequences in $\Chfd(\R)$.
We have 
$$Ev_{1}\circ F=A, \ \ Ev_{2}\circ F=B, \ \  Ev_{12}\circ F=C$$
In this case from (\ref{ADD THETA2 EQ}) we obtain a diagram
\begin{equation}
\label{ADD THETA ABC EQ}
\xymatrix{
k\circ |B| \ar[rr]^(.4){\mho}\ar[d]_{\Theta^{2}}&  & 
k\circ |A|+ k\circ |C|\ar[d]^{\Theta^{1}+\Theta^{12}} \\
k\circ |H|\circ |B|\ar[rr]_(.4){\mho^{H}}& & 
k\circ |H|\circ |A|+ k\circ |H| \circ |C| \\
}
\end{equation}
Vertices of this diagram represent a maps $|\DD|\to K(\R)$ and edges give homotopies of such maps.
Moreover, Lemma \ref{ADD THETA LEMMA} shows this diagram can be filled by a homotopy 
of homotopies. We apply this observation in the proof of Theorem  \ref{ADD THM} as follows.  
Let $p\colon E\to B$ be a bundle with a unipotent splitting as in the statement (\ref{ADD THM}). 
Take $\DD:=\simp$ and let $F\colon \simp\to \SS[2]\Chfd(\R)$
be the functor which assigns to $\sigma\in \simp$ the short exact sequence (\ref{ADD SES EQ}).
Then the diagram (\ref{ADD THETA ABC EQ}) is the top square of the diagram (\ref{ADD HOM HOM 3 EQ}).
\end{nn}

\vskip 3mm
\subsection{Tensor products}
Let  $T\in \Chfd(\R)$ be a chain complex with trivial differentials. Consider the functor
$$\otimes T\colon \Chfd(\R)\to \Chfd(\R)$$
which maps a chain complex $C$ to $C\otimes T$. This is an exact functor, so it induces 
a map  
$$K(\otimes T)\colon K(\R)\to K(\R)$$
Restricting the functor $\otimes T$ to the category $\SS[1]\Chfd(\R)$ and then passing to 
the nerve we also get a map 
$$|\otimes T|\colon |\SS[1]\Chfd(\R)|\to |\SS[1]\Chfd(\R)|$$
Notice that 
$$k\circ |\otimes T| =K(\otimes T)\circ k \ \ \text{and} \ \   k\circ |H| \circ |\otimes T|= K(\otimes T)\circ k \circ |H|$$
As a consequence the homotopy $K(\otimes T)\circ k\circ |H|\simeq K(\otimes T)\circ k$ can be obtained in two 
different ways:
\begin{itemize}
\item[--] using the map $K(\otimes T)\circ \Theta$, or
\item[--] using the map $\Theta \circ (|\otimes T|\times \id_{I})$
\end{itemize}
where $\Theta$ is the homotopy given by Lemma \ref{HOMOLOGY LEMMA}.

\begin{lemma}
\label{TENSOR LEMMA}
There exists a homotopy of homotopies between $K(\otimes T)\circ \Theta$ and  $\Theta \circ (|\otimes T|\times \id_{I})$.
\end{lemma}

\begin{proof}
Let $C\in \Chfd(\R)$. The homotopy $K(\otimes T)\circ \Theta$ is obtained 
applying Waldhausen's pre-additivity theorem 
to short exact sequences
$$(Q_{q}C)\otimes T\to (P_{q}C)\otimes T \to (P_{q-1}C)\otimes T$$
while the homotopy $\Theta\circ (|T|\times \id_{I})$ comes from the pre-additivity theorem 
applied to short exact sequences
$$Q_{q}(C\otimes T)\to (P_{q}C \otimes T) \to P_{q-1}(C\otimes T)$$ 

Assume that the chain complex $T$ is non-zero in only one grading $n$. 
In this case the homotopy of homotopies
between $K(\otimes T)\circ \Theta$ and $\Theta\circ (|T|\times \id_{I})$ comes from isomorphisms of short 
exact sequences 
$$
\xymatrix{
(Q_{q}C)\otimes T\ar[r] \ar[d]_{\cong} & (P_{q}C)\otimes T \ar[r]\ar[d]_{\cong} 
& (P_{q-1}C)\otimes T\ar[d]^{\cong} \\
Q_{q+n}(C\otimes T)\ar[r] & (P_{q+n}C \otimes T) \ar[r] & P_{q+n-1}(C\otimes T) \\
}
$$
If $T$ is an arbitrary complex with trivial differentials then $T$ is a direct sum of complexes concentrated in 
a single grading. Using this observation we can reduce the statement of the lemma 
to the special case considered above. 
\end{proof}

\begin{nn}
\label{TOP SQ PROP}
Let $\DD$ be a small category and let 
$$F, G\colon \DD \to \SS[1]\Chfd(\R)$$
be functors. Assume that for a some complex $T\in \Chfd(\R)$ with trivial differentials 
we have a natural transformation of functors 
$$\alpha\colon F\Rightarrow \otimes T\circ G$$
This induces a homotopy $|\alpha|\colon |\DD|\times I \to |\SS[1]\Chfd(\R)|$ between the maps 
$|F|$ and $|\otimes T\circ G|$. Consider the diagram 
\begin{equation*}
\xymatrix{
k\circ |F| \ar[rr]^(.45){k\circ |\alpha|} \ar[d]_{\Theta\circ (|F|\times \id_{I})}& 
&k\circ  |\otimes T\circ G|\ar@{=}[r] \ar[d]_{\Theta\circ (|\otimes T\circ G|\times \id_{I})} & 
K(\otimes T)\circ k\circ |G| \ar[d]_{K(\otimes T)\circ \Theta\circ (|G|\times \id_{I})}\\
k\circ |H|\circ |F| \ar[rr]_(.45){k\circ |H(\alpha)|} & 
& k\circ |H|\circ |\otimes T\circ G| \ar@{=}[r] & K(\otimes T)\circ k\circ |H|\circ |G|\\
}
\end{equation*}
Each vertex of the diagram represents a map $|\DD|\to K(\R)$, and each edge stands for a homotopy 
of such maps. We claim that there exists a homotopy of homotopies between the concatenation of 
$k\circ|\alpha|$ with $K(\otimes T)\circ \Theta\circ (|G|\times \id_{I})$ and concatenation of 
$\Theta\circ (|F|\times \id_{I})$ with $k\circ |H(\alpha)|$. Indeed, directly from the construction of the homotopy 
$\Theta$ one can see that the left square in the diagram can be filled by a homotopy of homotopies. 
Also,  by Lemma \ref{TENSOR LEMMA} we have a homotopy of homotopies filling the right square.

We can apply this observation in the context of the proof of Theorem \ref{SEC TRANSFER THM} as follows. 
Let $p\colon E\to B$ be a smooth bundle and let  $q\colon D\to E$ be a Leray-Hirsch bundle with fiber
$F_{q}$. Take $\DD:=\simp$. Let  $F:=C_{pq}$ and $G:= C_{p}$ be the chain complex functors 
defined in (\ref{LIN NN}). Also, let $T:= H_{\ast}(F_{q})$. We have a natural 
transformation 
$$\alpha\colon C_{pq}\Rightarrow \otimes H_{\ast}(F_{q})\circ C_{p}$$
given by the Leray-Hirsch quasi-isomorphism. Consider the diagram (\ref{SEC HOMOT DIAG}). 
The homotopy $\mu_{q}\circ (p^{!}\times \id_{I})$ in that diagram coincides with $k\circ |\alpha|$
while the homotopies $\omega^{(1)}_{pq}$ and $\omega^{(1)}_{p}\otimes H_{\ast}(F_{q})$ can be identified
with, respectively, $\Theta\circ (|F|\times \id_{I})$ and $K(\otimes T)\circ \Theta\circ (|G|\times \id_{I})$. 
Finally, the homotopy $\alpha(q)$ in (\ref{SEC HOMOT DIAG}) is the same as $k\circ |H(\alpha)|$.
As a consequence the homotopy of homotopies described above fills the top square in the diagram 
(\ref{SEC HOMOT DIAG}) in the proof of Theorem \ref{SEC TRANSFER THM}.

\end{nn}

\bibliographystyle{plain}
\bibliography{axioms}

\end{document}